\documentclass{amsart}
\usepackage{hyperref, cleveref, amsthm, amssymb, color, mathtools, thmtools}

\usepackage{comment}
\usepackage{enumerate}

\usepackage{tikz, tcolorbox}
\usetikzlibrary{cd}

\input xy
\xyoption{all}

\theoremstyle{plain}
\newtheorem{theorem}{Theorem}[section]
\newtheorem{proposition}[theorem]{Proposition}
\newtheorem{corollary}[theorem]{Corollary}
\newtheorem{lemma}[theorem]{Lemma}

\theoremstyle{definition}
\newtheorem{remark}[theorem]{Remark}

\numberwithin{equation}{section}

\crefname{theorem}{Theorem}{Theorems}
\crefname{proposition}{Proposition}{Propositions}
\crefname{section}{Section}{Sections}
\crefname{lemma}{Lemma}{Lemmas}

\newcommand{\N}{\mathbb{N}}

\newcommand{\R}{\mathbb{R}}

\newcommand{\dist}{\textrm{dist}\,}
\newcommand{\spn}{\textrm{span}\,}
\newcommand{\conv}{\textrm{conv}\,}
\newcommand{\ran}{\textrm{ran}\,}
\newcommand{\rank}{\textrm{rank}\,}
\newcommand{\vr}{\varepsilon}
\newcommand{\ball}{\mathrm{B}}
\newcommand{\sphere}{\textrm{S}}

\newcommand{\codim}{\mathrm{codim} \, }

\newcommand{\marg}[1]{\marginpar{\tiny #1}}     
\newcommand{\timur}[1]{{\textcolor{blue}{{\bf T.O:} #1}}}
\newcommand{\ana}[1]{{\textcolor{blue}{{\bf A.I:} #1}}}

\newcommand{\id}{{\textrm{id}}}

\DeclareSymbolFont{bbold}{U}{bbold}{m}{n}
\DeclareSymbolFontAlphabet{\mathbbold}{bbold}

\DeclareMathOperator{\FVL}{FVL}
\DeclareMathOperator{\FBL}{FBL}
\DeclareMathOperator{\FBLp}{FBL^{\mathit{(p)}}}

\newcommand{\Hp}{H_{w^*}^{p}}
\newcommand{\nip}{I_{w^*}^{p}}
\newcommand{\Ip}{\overline{I}_{w^*}^{p}}

\begin{document}

\title[Upper bound properties via operators]{Upper bound properties of free and related Banach lattices via operators}

\author{Anastasiia Ianina and Timur Oikhberg}


\address{Department of Mathematics, University of Illinois at Urbana-Champaign, Urbana, IL}
 \email{aianina2@illinois.edu}
 \email{oikhberg@illinois.edu}

\thanks{This work was partially supported by UIUC CRB award RB26177. The authors are grateful to M.~Taylor, P.~Tradacete, and V.~Troitsky for many stimulating conversations. 
}


\begin{abstract}
The free $p$-convex Banach lattice on a Banach space $X$ can be represented as a space of functions on the unit ball of $X^*$.
In this way, it gives rise to certain related (larger) lattices. To investigate such lattices, we initiate the study of functions induced by operators into $X$ and by subsets of $X$.
This new tool permits us to (i) determine whether the lattices in question possess, or fail, properties involving upper bounds of upward directed sets -- namely, the Fatou property, and the related property of monotonic boundedness; (ii) investigate the embeddings between spaces in question.
We uncover connections between the Fatou-like properties of $\FBLp[X]$ and the Gelfand-Phillips property of order $p$ of $X$. Along the way, we produce examples of spaces possessing the weak Gelfand-Phillips property of order $p$, but not the Gelfand-Phillips property of order $p$, thus answering a question of Ardakani and Miranda.
\end{abstract}

\maketitle

\tableofcontents

\section{Introduction and preliminaries}\label{sec:intro}

The theory of free Banach lattices has been advancing rapidly in the past several years. A free $p$-convex Banach lattice $\FBLp[X]$, over a Banach space $X$, is the (necessarily unique) Banach lattice $L$ such that (i) there exists an isometric embedding $\phi_{X,p}: X \to L$, and (ii) every bounded linear operator $T : X \to Z$, where $Z$ is a Banach lattice with $p$-convexity constant $1$, extends to a lattice homomorphism $\widehat{T}: L \to Z$, with $\|\widehat{T}\| = \|T\|$, as illustrated by the diagram below:
$$
\xymatrix{
 \FBLp[X] \ar@{-->}[dr]^{\widehat{T}}  &   \\
  X \ar[r]^{T} \ar@{^{(}->}[u]^{\phi_{X,p}}  & Z } 
$$
It is known (see, e.g.~\cite{OiTTT} or \cite{GSLTT}) that $\FBLp[X]$ admits the following functional representation. Consider $C(\ball(X^*))$, the space of weak$^*$-continuous functions on the closed unit ball $\ball(X^*)$.
For $x \in X$, define $\phi_{X,p}(x) = \delta_x : x^* \mapsto \langle x^*,x\rangle$.
The \emph{free vector lattice} on $X$ ($\FVL[X]$) is the vector sublattice of $C(\ball(X^*))$ generated by $\phi_{X,p}(X)$.
For $f \in C(\ball(X^*))$, define the ``$p$-summing'' norm $\|f\|_p$: 
\begin{equation}
\|f\|_p = \inf \big\{ C : \forall x_1^*, \ldots, x_n^* \in X^*, (\sum_{i=1}^n |f(x_i^*)|^p)^{1/p} \leq C \|(x_i^*)_{i=1}^n\|_{p, {\textrm{weak}}} \big\} .
\label{def of p-norm}
\end{equation}
Throughout this paper, for $z_1, \ldots, z_n \in Z$ we write $\|(z_i)_{i=1}^n\|_{p, {\textrm{weak}}} = \sup \big\{ \| \sum_i \alpha_i z_i \| : \|(\alpha_i)_i\|_{\ell_{p'}} \leq 1 \big\}$, and $1/p + 1/p' = 1$. Then $\FBLp[X]$ is the closure of $\FVL[X]$ in the norm $\| \cdot \|_p$. 

This functional representation leads naturally to the following larger lattices: 
\begin{itemize}
 \item $H_{w^*}[X]$ is the linear space of positively homogeneous functions $X^* \to \R$, weak$^*$-continuous on bounded sets;
 \item $\Hp[X]$ is the space of all $f \in H_{w^*}[X]$ with $\|f\|_p < \infty$;
 \item $I_{w^*}[X]$ (respectively, $\nip[X]$) is the ideal generated in $H_{w^*}[X]$ by $(\delta_x)_{x \in X}$ (respectively, by $\FBLp[X]$);
 \item $\Ip[X]$ is the closure of $I_{w^*}[X]$ (equivalently, of $\nip[X]$) in $\Hp[X]$.
\end{itemize}
When $p = \infty$, \cite[Proposition 2.2]{OiTTT} implies that the spaces $H_{w^*}^\infty[X]$, $\overline{I}_{w^*}^\infty[X]$, and $\FBL^{(\infty)}[X]$ all coincide with the space of positively homogeneous weak$^*$-continuous functions on $\ball(X^*)$, and $\| \cdot \|_\infty$ equals the $\sup$ norm. In general, $\FBLp[X] \subset \nip[X] \subset \Ip[X] \subset \Hp[X]$, and we shall see later that all the inclusions may be proper. See \cite{OiTTT} for a thorough study of free Banach lattices, and \cite{LaTr} for more information on the related function spaces introduced above.

The question of linking properties of a Banach space $X$ with those of the free Banach lattice built on it was raised in \cite{OiTTT}, where several such connections were established. The present work can be viewed as a step in this direction. In particular, we relate certain order-theoretic properties of free Banach lattices to the (weak) Gelfand-Phillips Property of order $p$ of an underlying Banach space (introduced in \Cref{general}). 
We also study natural embeddings between free and related Banach lattices, continuing the investigations initiated in \cite{LaTr}. Our approach to both problems is based on operators with range in $X$, as their adjoints naturally give rise to positively homogeneous functions on $\ball(X^*)$.

We focus on two order-theoretic properties of Banach lattices: the Fatou property and monotonic boundedness.
We say that a Banach lattice $Z$ has the \emph{$\lambda$-Fatou property} if, for every increasing net $(z_{\alpha}) \subset Z_+$ with supremum $z_0$, one has $\|z_0\| \leq \lambda \sup_{\alpha} \|z_{\alpha}\|$. When $\lambda = 1$, we simply say that $Z$ has the \emph{Fatou property}.

Following e.g.~\cite[Definition 4.2]{DLOT} or \cite[Section 10]{Taylor}, we say that a Banach lattice $Z$ is \emph{monotonically bounded} if every norm-bounded increasing net in $Z_+$ has an upper bound. For $\lambda \geq 1$, we say that $Z$ is \emph{$\lambda$-monotonically bounded} if every such net $(z_{\alpha})$ admits an upper bound $z_0$ satisfying $\|z_0\| \leq \lambda \sup_\alpha \|z_\alpha\|$. We show in \Cref{m-b} that every monotonically bounded Banach lattice is $\lambda$-monotonically bounded for some $\lambda$. Clearly, $\lambda$-monotonic boundedness implies the $\lambda$-Fatou property. The converse fails in general, as discussed in \Cref{m-b}. Restricting these definitions to increasing sequences gives the corresponding countable properties.

The Fatou property and monotonic boundedness in free Banach lattices were first studied in \cite{ART}, where it was shown that $\FBL[\ell_1(A)]$ is $1$-monotonically bounded for any set $A$, while $\FBL[L_1[0, 1]]$ fails the Fatou property. In \cite{ABT}, the former result was extended to the $p$-convex setting and $\lambda$-monotonic boundedness of $\FBL[X]$ was shown to preclude $X$ from containing copies of $c_0$ or $L_1[0, 1]$.
\begin{remark}
In \cite{ART,ABT}, the term ``strong Nakano'' is used to refer to $1$-monotonic boundedness. 
In the literature, one also encounters a notion intermediate between $\lambda$-monotonic boundedness and $\lambda$-Fatou property, known as the $\lambda$-Nakano property (see \cite{Wick07,ABT}). A Banach lattice $Z$ has this property if every upward directed, order-bounded subset of $\ball(Z)_+$ admits an upper bound of norm at most $\lambda$. Although we do not study this property explicitly, our results imply that all three properties are equivalent for free $p$-convex Banach lattices. 
\end{remark}

We begin by gathering, in \Cref{m-b}, essential facts about monotonic boundedness and the $\lambda$-Fatou property.
In \Cref{general}, we introduce functions on $X^*$ induced by operators $X^* \to Z$, and by subsets of $X^{**}$, and characterize when they belong to the lattices in question. Along the way, we recall the definitions of $p$-compact operators, the Gelfand-Phillips property of order $p$ ($p$-GPP), and its weak counterpart (weak $p$-GPP), all of which play a central role in our work. Throughout this paper, we assume $p \in [1,\infty)$, unless specified otherwise.

Having laid this groundwork, we show in \Cref{+nakano} that the $p$-GPP of a Banach space $X$ implies $1$-monotonic boundedness of $\FBLp[X]$ and $\Ip[X]$. The corresponding conclusion for $\Hp[X]$ requires only the weak $p$-GPP.
In \Cref{operators:general}, we further develop our operator approach and use it to detect failure of $\lambda$-Fatou property and monotonic boundedness in free and related lattices. In particular, \Cref{duals and FBLp} and \Cref{duals and Hp} yield the following characterizations: 
\begin{itemize}
\item $\FBLp[X]$ is monotonically bounded if and only if $\FBLp[X]$ has the $\lambda$-Fatou property for some $($equivalently, any$)$ $\lambda$ if and only if $X$ has the $p$-GPP;
\item $\Hp[X]$ is monotonically bounded if and only if $\Hp[X]$ has the $\lambda$-Fatou property for some $($equivalently, any$)$ $\lambda$ if and only if $X$ has the weak $p$-GPP.
\end{itemize}
Similar results hold for $\Ip[X]$.
The case $p=\infty$ is different: $\FBL^{(\infty)}[X]$ fails the $\lambda$-Fatou property unless $X$ is finite-dimensional. 

In \Cref{s:embeddings}, we apply our operator approach to the question of distinguishing between free and related lattices raised in \cite{LaTr}. In particular, in \Cref{FBLp into Ip}, we prove that $\FBLp[X]$ is not a regular sublattice of $\nip[X]$ unless $X$ is finite-dimensional. This strengthens \cite[Theorem 4.3]{LaTr}, which requires additional assumptions on $X$ to show that $\FBLp[X]$ is a proper subspace of $\Ip[X]$. In \Cref{Ip into Hp}, we prove that $\nip[X]$ is closed in $\Hp[X]$ if and only if $\nip[X] = \Hp[X]$, and both assertions are equivalent to $X$ having the $p$-GPP. This answers another question raised in \cite{LaTr}. In contrast, \Cref{sup-norm} establishes that, for every Banach space $X$, the $\|\cdot\|_p$-unit balls of $I_{w^*}[X]$ and $\Ip[X]$ are $\|\cdot\|_{\infty}$-dense in the $\|\cdot\|_p$-unit ball of $\Hp[X]$. 

Finally, in \Cref{examples}, we give various examples of Banach spaces with and without the (weak) $p$-GPP. Notably, we identify several spaces which have the weak $p$-GPP but fail the $p$-GPP, answering a question left open in \cite{AM}.
\begin{remark}
An earlier version of this paper used the tools discussed here to the ``regular sublattice'' problem. More specifically, we showed that, for a class of pairs $(X,Y)$ ($X$ being a subspace of $Y$), the canonical image of $\FBLp[X]$ in $\FBLp[Y]$ is not a regular sublattice.
Shortly after that version was posted on Arxiv, \cite{PMT} appeared, resolving the problem completely.
\end{remark}

Throughout, we rely on standard notation and facts. In particular, we denote by $\ball(X)$ and $\sphere(X)$ the closed unit ball and the unit sphere of $X$, respectively, and by $B(X, Y)$ the space of bounded linear operators from $X$ to $Y$; for $p \in [1, \infty]$, $p'$ shall denote the conjugate exponent of $p$, i.e., $1/p + 1/p' = 1$. The required background on Banach spaces can be found, for instance, in \cite{AK}. For information about Banach lattices, we refer the reader to e.g.~\cite{M-N}. 

\section{Introduction to upper bound properties}\label{m-b}

We gather the essential facts about the $\lambda$-Fatou property and monotonic boundedness. For applications of monotonic boundedness to regular operators, see \cite[Section 4]{DLOT}, and for other related properties, see \cite{Taylor}.

\begin{proposition}
 If a Banach lattice $X$ is monotonically bounded (respectively, countably monotonically bounded), then it is $\lambda$-monotonically bounded (respectively, countably $\lambda$-monotonically bounded) for some $\lambda$.
\end{proposition}

\begin{proof}
 We prove the ``monotonically bounded'' case, the countable setting is handled similarly.
 Suppose, for the sake of contradiction, that $X$ is monotonically bounded, but not $\lambda$-monotonically bounded for any $\lambda$. Then, for every $n$, there exists an increasing net $(z^{(n)}_\alpha)_{\alpha \in {\mathcal{A}}_n} \subset X_+$ so that $\sup_{\alpha \in {\mathcal{A}}_n} \|z^{(n)}_\alpha\| < 2^{-n}$, yet $\|z_0\| > 2^n$ whenever $z_0$ is an upper bound for this net.

 Let $\mathcal{A}$ be the set of all finite sequences $(\alpha_1, \ldots, \alpha_m)$, with $m \in \N$ and $\alpha_i \in {\mathcal{A}}_i$. Equip $\mathcal{A}$ with the following partial order: $(\alpha_1, \ldots, \alpha_m) \prec (\beta_1, \ldots, \beta_n)$ if $m \leq n$, and $\alpha_i \leq \beta_i$ for $1 \leq i \leq m$. Then the net $z_{(\alpha_1, \ldots, \alpha_m)} = \sum_{i=1}^m z^{(i)}_{\alpha_i}$ is upward directed, and norms of its elements are bounded by $1$. On the other hand, it cannot have an upper bound.
\end{proof}

Recall that a Banach lattice $X$ is called \emph{monotonically complete} if every norm-bounded increasing net in $X$ has a supremum. For properties of such lattices, see \cite[Section 2.4]{M-N}.

\begin{proposition}\label{connections}
Let $X$ be a Banach lattice.
 \begin{enumerate}
  \item If $X$ is $\lambda$-monotonically bounded, then it has the $\lambda$-Fatou property. The same holds for the ``countable'' versions of these properties.
  \item If $X$ is monotonically complete, then it is monotonically bounded. Moreover, in this case, $X$ is $\lambda$-monotonically bounded if and only if it has the $\lambda$-Fatou property.
 \end{enumerate}
The converse statements are false.
\end{proposition}

\begin{proof}
 The forward statements are straightforward, see also \cite[Proposition 2.2]{ABT}. It remains to provide counterexamples.

 (1) Let $X$ be the space of continuous functions on $[0,1]$ which vanish at $0$. By \cite{Wick07}, $X$ has the Fatou property. However, it is easy to see it is not monotonically bounded.
 
 (2) Clearly, $X = C[0,1]$ works. 
\end{proof}

By \cite[Theorem 2.4.19]{M-N}, any dual of a Banach lattice is monotonically complete and has the Fatou property; consequently, it is also 1-monotonically bounded.
On the other end of the spectrum, an AM-space is 1-monotonically bounded if and only if it is lattice isometric to a $C(K)$-space, for some Hausdorff compact $K$, see \cite{ABT}.

\begin{remark}
    If $\dim X > 1$, then $\FBLp[X]$ is not order complete by \cite[Proposition 2.11]{OiTTT}, and hence it fails to be monotonically complete. In contrast, monotonic boundedness and the $\lambda$-Fatou property can still hold in free $p$-convex Banach lattices, as we shall see below.
\end{remark}

For separable lattices, the argument of \cite[Lemma 2.5]{ABT} yields the following.

\begin{proposition}\label{sep mon bound}
Let $X$ be a separable Banach lattice, then
\begin{enumerate}
    \item $X$ is $\lambda$-monotonically bounded if and only if it is countably $\lambda$-monotonically bounded.
    \item $X$ has the $\lambda$-Fatou property if and only if it has the countable $\lambda$-Fatou property.
\end{enumerate}
\end{proposition}

Recall that for vector lattices $E$ and $F$ and an injective lattice homomorphism $j: E \to F$, $j(E)$ is called a \emph{regular sublattice} of $F$ if, for every set $A \subset E$ with supremum $a$ in $E$, $ja$ is the supremum of $j(A)$ in $F$. 

\begin{remark}
    By \cite[Proposition 2.9]{ABT}, the $\lambda$-Fatou property passes down to regular sublattices. Monotonic boundedness need not pass even to closed ideals as witnessed by $c_0$ and $\ell_{\infty}$.
\end{remark}
We should also mention \cite[Proposition 12.3]{dPWi}: if a monotonically bounded Banach lattice $Y$ is embedded into a Banach lattice $X$ as an ideal, then $Y$ is a projection band in $X$. 

\section{Functions generated by operators and sets}\label{general}

Our exposition relies heavily on the ideals $\Pi_p$, $I_p$, and $N_p$ of \emph{$p$-summing}, \emph{$p$-integral}, and \emph{$p$-nuclear} operators; the corresponding norms are denoted by $\pi_p( \cdot )$, $\iota_p( \cdot )$, and $\nu_p( \cdot )$, respectively; see \cite{DJT} for more information. Other relevant ideals are introduced below.

We shall often use a special class of functions generated by operators: for $T : X^* \to Z$, we define $\phi_T(x^*) = \|T x^*\|$. It is easy to see that $\phi_T(x^*) = \sup \{ | \langle x^{**}, x^* \rangle | : x^{**} \in T^{*} (\ball(Z^*))\}$.
By \eqref{def of p-norm}, $\|\phi_T\|_p = \pi_p(T)$. 

A closely related class consists of set-generated functions. For a bounded set $A \subset X^{**}$, define $h_A(x^*) = \sup_{x^{**} \in A} |\langle x^{**}, x^* \rangle|$. For $A \subset X$, we use the same notation, viewing $X$ as ``canonically embedded'' into $X^{**}$. A connection of functions $h_A$ to operator-induced functions is explained by the following:

\begin{lemma}\label{phi versus h}
\begin{enumerate}
\item 
For any Banach space $Z$ and $T \in B(X^*,Z)$, $\phi_T = h_{T^*(\ball(Z^*))}$. Moreover, for $S \in B(Z,X)$, $\phi_{S^*} = h_{S(\ball(Z))}$.
\item 
For any $A \subset X$, $h_A = \phi_{R^*}$, where $R \in B(\ell_1(A),X)$ takes the Kronecker function $\delta_a$ to $a$.
\end{enumerate}
\end{lemma}

\begin{remark}\label{associated operator}
We call $R$ constructed above the \emph{operator associated with $A$}. Note that, if $A$ is absolutely convex, then $R(\ball(\ell_1(A))) \subset \overline A$. 
\end{remark}

The following result can be obtained by combining \cite[Section 7]{JLTTT} (see also \cite[Proposition 2.6]{OiGEO}) with Pietsch Factorization Theorem.

\begin{proposition}\label{FBLp norm}
For $p \in [1,\infty)$, and $\phi \in H_{w^*}[X]$, the following statements are equivalent:
\begin{enumerate}
\item $\phi \in \Hp[X]$, with $\| \phi \|_p \leq 1$.
\item There exist a probability measure $\mu$ and an operator $T : X^* \to L_p(\mu)$ so that $\iota_p(T) \leq 1$, and $|\phi| \leq \phi_T$.
\item There exist a Banach space $Y$ and an operator $T : X^* \to Y$ so that $\pi_p(T) \leq 1$, and $|\phi| \leq \phi_T$.
\end{enumerate}
\end{proposition}

The following lemma characterizes operator-generated functions belonging to $H_{w^*}[X]$: 

\begin{lemma}\label{weak* continuity}
For $S \in B(Z,X)$, the following are equivalent:
\begin{enumerate}
   \item $S$ is compact.
   \item $S^*$ is weak$^*$ to norm continuous on bounded subsets of $X^*$.
   \item The restriction of $\phi_{S^*}$ to $\ball(X^*)$ is weak$^*$ continuous at $0$.
\end{enumerate}
\end{lemma}

\begin{proof}
$(1) \Rightarrow (2)$:
Let $q : \ell_1(I) \to Z$ be a quotient map, where $I$ is a sufficiently large index set. The space $\ell_\infty(I) = \ell_1(I)^*$ has the (Metric) Approximation Property, hence, by \cite[Theorem 1.e.5]{LT1}, there exists a net of finite rank maps $T_u : \ell_1(I) \to X$, converging to $Sq$ in norm. Clearly $T_u^*$ is weak$^*$ to norm continuous, hence so is $q^* S^*$. Now recall that, for $x^* \in X^*$, $\|S^* x^*\| = \|q^* S^* x^*\|$. 

$(2) \Rightarrow (3)$ is straightforward.

$(3) \Rightarrow (1)$: If $S$ is as in (3), then for any $\vr > 0$ there exist $x_1, \ldots, x_n \in X$  so that $\|S^* x^*\| < \vr$ whenever $x^* \in \ball(X^*)$ satisfies $\vee_i |\langle x^*, x_i \rangle| < 1$.
Then in particular, $\|S^*|_{E^\perp}\| \leq \vr$, where $E = \spn[x_1, \ldots, x_n]$. The former quantity equals $\|S^* q_E^*\| \leq \vr$; here $q_E : X \to X/E$ is the quotient map, and $q_E^* : E^\perp \to X^*$ is the canonical embedding.
So for every $\vr > 0$ there exists a finite dimensional $E \subset X$ so that $\|q_E S\| \leq \vr$. This implies that $S$ is compact.
\end{proof}

Combining the above with the definition of the norm in $\Hp[X]$, we obtain:

\begin{corollary}\label{compact_adjoint}
    For $T \in B(Z, X)$, $\phi_{T^*} \in \Hp[X]$ if and only if $T$ is compact and $T^* \in \Pi_p(X^*, Z^*)$; moreover  $\|\phi_{T^*}\|_p = \pi_p(T^*)$.
\end{corollary}

For finite rank operators, more can be said:

\begin{lemma}\label{fin rank ops}
Suppose $X$ is a Banach space, $E$ is finite dimensional, and $u \in B(X^*,E)$ is weak$^*$ continuous. Then $\phi_u \in \FBLp[X]$, and $\|\phi_u\| = \pi_p(u)$. Moreover, for every $\delta > 0$ there exists $\psi \in \FVL[X]_+$ so that $(1-\delta) \psi \leq (1-\delta) \phi_u \leq \psi$.
\end{lemma}

\begin{proof}
By scaling, we can assume that $\pi_p(u) = 1$. It is straightforward that $\phi_u$ is weak$^*$ continuous on bounded sets, and $\|\phi_u\|_{p} = \pi_p(u) = 1$. It remains to show that $\phi_u$ lies in $\FBLp[X]$ -- that is, for any $\vr > 0$ there exists $\psi \in \FVL[X]$ so that $\|\phi_u - \psi\| < \vr$.
To this end, find $N$ and a contraction $v \in B(E,\ell_\infty^N)$, so that $\|ve\| \geq 1 - \delta = 1 - \vr/2$ for any $e \in E$ of unit norm. Write $ve = (\langle e_j^*,e \rangle)_{j=1}^N$. Let $\psi(x^*) = \|vux^*\|$, then clearly $\psi = \vee_{j=1}^N \big| \delta_{u^* e_j^*} \big|$.
The weak$^*$-continuity of $u$ implies $u^* e_j^* \in X$, so $\psi \in \FVL[X]$. Also, $\psi \leq \phi_u \leq (1-\delta)^{-1} \psi$, hence $0 \leq \phi_u - \psi \leq \delta(1-\delta)^{-1} \phi_u$, giving us the desired estimate for $\|\phi_u - \psi\|$.
\end{proof}

\begin{corollary}\label{nuclear in FBL}
If $T: X^* \to Y$ is approximable in the $p$-summing norm by a sequence of weak$^*$ continuous finite-rank operators $u_n : X^* \to Y$, then $\phi_T \in \FBLp[X]$ with $\|\phi_T\| = \pi_p(T)$.
\end{corollary}

We shall see later that the conclusion $\phi_T \in \FBLp[X]$ may fail if $T$ is merely approximable (in the operator norm).

To perform a finer analysis, recall that an operator $T : Z \to X$ is called \emph{quasi-$p$-nuclear} ($1 \leq p < \infty$) if there exists $(z_i^*) \in \ell_p(Z^*)$ so that $\|Tz\| \leq \big( \sum_i | \langle z_i^*, z \rangle |^p \big)^{1/p}$ for any $z \in Z$.
If $j : X \to W$ is an isometric embedding of $X$ into a $1$-injective Banach space $W$ (such as $\ell_\infty(\Gamma)$, for some set $\Gamma$), then $T : Z \to X$ is quasi-$p$-nuclear if and only if $jT$ is $p$-nuclear.

Further (see e.g.~\cite{DPS}) a subset $A$ of a Banach space $X$ is called \emph{relatively $p$-compact} (with $1 \leq p < \infty$) if there exists $(x_i) \in \ell_p(X)$ such that $A \subset p-\conv(x_i)$ -- that is, for each $a \in A$, there exists a sequence $(\zeta_i(a)) \in \ball(\ell_{p'})$ so that $a = \sum_i \zeta_i(a) x_i$.
This definition is motivated by Grothendieck's characterization of relatively compact sets as subsets of closed convex hulls of null-sequences.

An operator $T : Z \to X$ is called \emph{$p$-compact} if $T(\ball(Z))$ is relatively $p$-compact. In this investigation, $p$-compact operators are significant, as they generate elements of $\FBLp[X]$. 

\begin{proposition}\label{p-compactness}
For an operator $T : Z \to X$, the following are equivalent:
\begin{enumerate}
\item $T$ is $p$-compact;
\item $T^*$ is quasi-$p$-nuclear;
\item $\phi_{T^*} \in \FBLp[X]$;
\item $\phi_{T^*} \in \nip[X]$.
\end{enumerate}
\end{proposition}

\begin{proof}
By \cite{DPS}, (1) $\Leftrightarrow$ (2) (the proof in \cite{DPS} relies on a result of O.~Reinov; for a neat exposition of a generalization of that result see \cite{Oja}), (3) $\Rightarrow$ (4) is straightforward, and (4) $\Rightarrow$ (2) follows from \cite[Proposition 2.6]{OiGEO}. It remains to show (1) $\Rightarrow$ (3).
Choose $(x_n) \in \ell_p(X)$ and, for each $z \in \ball(Z)$, find $\alpha(z)\in \ball(\ell_{p'})$
such that $Tz = \sum_n \alpha_n(z) x_n$. Define $S, S_N:\ell_1(\ball(Z)) \to X$ by $S \delta_z = T z$ and $S_N \delta_z = \sum_{n \leq N} \alpha_n(z) x_n$.
Clearly, $\|S^*x^*\| = \|T^*x^*\|$ and
$$
\|(S^* - S_N^*)x^*\| = \sup_{z \in \ball(Z)} |\sum_{n > N} \alpha_n(z) x^*(x_n)| \leq (\sum_{n > N} |x^*(x_n)|^p)^{1/p},
$$ 
by H\"older's inequality. Hence, $\pi_p(S^* - S_N^*) \leq (\sum_{n > N}\|x_n\|^p)^{1/p} \to 0$. So, $\phi_{T^*} = \phi_{S^*} \in \FBLp[X]$ by \Cref{nuclear in FBL}.
\end{proof}

To formulate a criterion for $p$-compactness, suppose $Z, X$ are Banach spaces, and $X_0 \subset X$ is dense. We say that $T \in B(Z,X)$ is \emph{dual $p$-persistent relative to $X_0$} if there exists $c > 0$ so that, for any $x_1, \ldots, x_n \in X_0$, we have $\pi_p(T^*|_{\cap_i x_i^\perp}) > c$.
Call $T$ \emph{dual $p$-persistent} if it is dual $p$-persistent relative to some dense set $X_0 \subset X$.

\begin{proposition}\label{quasi p-nuc criteria}
For an operator $T : Z \to X$, the following statements are equivalent:
\begin{enumerate}
    \item $T$ is not $p$-compact.
    \item There exists $c > 0$ so that, for any $E \subset X^*$ of finite codimension, we have $\pi_p(T^*|_{E}) \geq c$.
    \item $T$ is dual $p$-persistent.
\end{enumerate}
\end{proposition}

\begin{proof}
The definition of dual $p$-persistence gives us (2) $\Rightarrow$ (3).

Next show $\lnot (1) \Rightarrow \lnot (3)$ (equivalently, (3) $\Rightarrow$ (1)). Let $X_0$ be a dense subset of $X$.
Suppose $T$ is $p$-compact -- that is, there exists a $p$-summable sequence $(x_i) \subset X$ so that $T(\ball(Z)) \subset p-\conv(x_i)$. Then $\|T^* x^*\| \leq \big( \sum_i | \langle x^*, x_i \rangle|^p \big)^{1/p}$. 
For any $\vr > 0$ find $N$ so that $\sum_{i>N} \|x_i\|^p < (\vr/2)^p$. For $1 \leq i \leq N$, find $y_i \in X_0$ so that $\|x_i - y_i\| < \vr/(2N^2)$, for every $i$. 

For $x^* \in E = \{y_1, \ldots, y_N\}^\perp \subset X^*$, 
$$
\|T^* x^*\| \leq \big( \sum_i | \langle x^*, x_i \rangle|^p \big)^{1/p} \leq \big(\|S x^*\|^p + \|R x^*\|^p \big)^{1/p} ,
$$
where
$$
S : X^* \to \ell_p^N : x^* \mapsto \big( \langle x^*, x_i \rangle \big)_{i \leq N} , \, \,
R : X^* \to \ell_p : x^* \mapsto \big( \langle x^*, x_i \rangle \big)_{i > N} .
$$
Then $\pi_p(T^*|_E) \leq \big(\pi_p(R|_E)^p + \pi_p(S|_E)^p\big)^{1/p}$.
By Pietsch Domination Theorem, 
$$
\pi_p(R|_E) \leq \big( \sum_{i>N} \|x_i\|^p \big)^{1/p} < \frac{\vr}2 .
$$
Also, 
$$
\pi_p(S|_E) \leq \|S|_E\| \rank (S|_E) \leq N \sum_i \|x_i|_E\| < \frac\vr2 .
$$
Thus, $\pi_p(T^*|_E) < \vr$; so, $T$ cannot be dual $p$-persistent relative to $X_0$.

To show (1) $\Rightarrow$ (2) (actually, $\lnot (2) \Rightarrow \lnot (1)$), let $q : \ell_1(\Gamma) \to Z$ be a quotient map. Clearly $T$ is $p$-compact if and only if $Tq$ is. Also, $q^* : Z^* \to \ell_\infty(\Gamma)$ is an isometric embedding, hence $T^*|_E$ is $p$-summing if and only if $q^* T^*|_E$ is (with the equality of norms). So, we can and do assume that $Z = \ell_1(\Gamma)$. Then $Z^* = \ell_\infty(\Gamma)$, hence the $p$-summing and $p$-integral norms of any operator into $Z^*$ coincide.


Suppose (2) fails -- that is, for any $n \in \N$ there exists $E_n \subset X^*$ so that $\dim(X/E_n) < \infty$, $\iota_p(T^*|_{E_n}) < 2^{-n}$. By replacing $E_n$ with $E_1 \cap \ldots \cap E_n$, we can clearly assume that $E_1 \supset E_2 \supset \ldots$.

As $Z^* = \ell_\infty(\Gamma)$ is injective, we can extend $T^*|_{E_1}$ to an operator $u_1 : X^* \to Z^*$, with $\iota_p(u_1) = \iota_p(T^*|_{E_1}) < 2^{-1}$.
Indeed, $Z^*$ is contractively complemented in its bidual, hence there exists a probability measure $\mu$ and operators $A : E_1 \to L_\infty(\mu)$ and $B : L_p(\mu) \to Z^*$ so that $T^*|_{E_1} = B \circ id \circ A$ ($id :L_\infty(\mu) \to L_p(\mu)$ is the formal identity), and $\|A\| \|B\| = \iota_p(T^*|_{E_1})$. Extend $A$ to $\widetilde{A} : X^* \to L_\infty(\mu)$, of the same norm as $A$. Then $u_1 = B \circ id \circ \widetilde{A}$ has the desired properties.

The operator $v_0 = T^* - u_1$ has finite rank (it vanishes on $E_1$), hence it is automatically $p$-nuclear. Note that $T^* = v_0 + u_1$.

Further, $u_1|_{E_2}$ extends to an operator $u_2 : X^* \to Z^*$, with $\iota_p(u_2) = \iota_p(T^*|_{E_2}) < 2^{-2}$. Then $v_1 = u_1 - u_2$ has finite rank, and $\iota_p(v_1) < 2^{-1} + 2^{-2} < 2^0$. As $Z^*$ has the Metric Approximation Property, \cite[Theorem 5.27]{DJT} implies $\nu_p(v_1) = \iota_p(v_1) < 2^0$. Clearly $T^* = v_0 + v_1  + u_2$.

Proceeding further in the same manner, for every $n$ we shall have $T^* = v_0 + v_1 + \ldots + v_{n-1} + u_n$, with $\iota_p(u_n) < 2^{-n}$, $v_k$ of finite rank for every $k$, and $\nu_p(v_k) < 2^{1-k}$.
Then $\big(\sum_{k=0}^{n-1} v_k\big)_n$ is a Cauchy sequence of $p$-nuclear operators; it must converge to some $p$-nuclear $S : X^* \to Z^*$ in $\nu_p( \cdot )$; consequently, $\lim_n \|S - \sum_{k=0}^{n-1} v_k\| = 0$. But $\|T^* - \sum_{k=0}^{n-1} v_k\| = \|u_n\| \to 0$, hence $S = T^*$. Consequently, $T^*$ is (quasi) $p$-nuclear, or equivalently, $T$ is $p$-compact.
\end{proof}

Recall that $A \subset X$ is called \emph{$p$-limited} ($1 \leq p < \infty$) if for every weakly $p$-summable sequence $(x_n^*) \subset X^*$ there exists $(\alpha_n) \in \ell_p$ such that $|\langle x_n^*, x\rangle| \leq \alpha_n$ for all $x \in A$ and $n \in \N$ (see \cite{DP-p-lim}).
By \cite[Theorem 3.1]{DP-p-lim}, $S \in B(Y,X)$ has $p$-summing adjoint if and only if $S(\ball(Y))$ is $p$-limited.

Following \cite{DP-p-lim}, we say that $X$ has the \emph{Gelfand-Phillips property of order $p$} (\emph{$p$-GPP} for short) if every $p$-limited subset of $X$ is relatively $p$-compact. This is precisely the $\mathcal C_p$ class considered in \cite{DP-p-comp}.

\begin{remark}
    The Gelfand-Phillips property of order $p$ should not be confused with the $p$-Gelfand-Phillips property introduced in \cite{FZ}. The latter property requires every limited weakly $p$-summable sequence to be norm null.
\end{remark}
Likewise, we say that $X$ has the \emph{weak Gelfand-Phillips property of order $p$} (\emph{weak $p$-GPP}) if every $p$-limited subset of $X$ is relatively norm compact. This property was introduced in \cite{Ghe} under the name $p$-limited relatively compact property; we follow the terminology of \cite{AM}. 

\begin{remark}\label{properties of GPOP}
Clearly, the $p$-GPP implies the weak $p$-GPP. It was asked in \cite{AM} whether these properties are equivalent. The results of \Cref{james tree} below provide a negative answer (at least for $p = 1$ and $p \geq 2$). 
\end{remark}

\begin{remark}\label{inherited}
Suppose $X$ is a subspace of $Y$, and $A \subset X$ is $p$-limited. Then $A$ is also a $p$-limited subset of $Y$: if $(y_n^*) \subset Y^*$ is weakly $p$-summable, then so is $(y_n^*|_X) \subset X^*$. Consequently, the weak $p$-GPP is inherited by closed subspaces. We do not know whether the same is true for the $p$-GPP.
\end{remark}

These properties admit useful operator characterizations. The following lemma was proved in \cite[Proposition 2.1]{DP-p-comp} (see also \cite{DP-p-lim}).
\begin{lemma}
    For a Banach space $X$, the following are equivalent:
\begin{enumerate}
    \item $X$ has the $p$-GPP;
    \item For any Banach space $Z$ and $T \in B(Z, X)$, $T$ is $p$-compact if and only if $T^*$ is $p$-summing.
    \item For any index set $A$ and $T \in B(\ell_1(A), X)$, $T$ is $p$-compact if and only if $T^*$ is $p$-summing.
\end{enumerate}
\end{lemma}
The corresponding characterization for weak the $p$-GPP is probably known, we include a short proof for completeness. 
\begin{lemma}\label{weak pGP}

For a Banach space $X$, the following are equivalent:
\begin{enumerate}
    \item $X$ has the weak $p$-GPP;
    \item For any Banach space $Z$ and $T \in B(Z, X)$, if $T^*$ is $p$-summing, then $T$ is compact.
    \item For any index set $A$ and $T \in B(\ell_1(A), X)$, if $T^*$ is $p$-summing, then $T$ is compact.
\end{enumerate}
\end{lemma}

\begin{proof}
(2) $\Rightarrow$ (3) is straightforward, while (1) $\Rightarrow$ (2) follows from \cite[Theorem 3.1]{DP-p-lim}. For (3) $\Rightarrow$ (1), let $A \subset X$ be $p$-limited. Define $T: \ell_1(A) \to X: \delta_a \mapsto a$. Since $\|T^*x^*\| = \sup_{a \in A}|x^*(a)|$ and $A$ is $p$-limited, $T^*$ maps weakly $p$-summable sequences to strongly $p$-summable sequences, i.e., $T^*$ is $p$-summing. By the assumption, $A$ is relatively compact.
\end{proof}

\begin{remark}\label{weak PG-p and l1}
    By \cite[Corollary 2.10]{AM}, a Banach space $X$ contains $\ell_1$ if and only if $X^*$ has the weak $p$-GPP for some (equivalently, any) $p \in [1,\infty)$.
\end{remark}

\section{Sufficient conditions for upper bound properties}\label{+nakano}

Throughout this section, we assume that $1 \leq p < \infty$. Our first goal is to determine under which assumptions $\Hp[X]$ is monotonically bounded.

\begin{theorem}\label{wRNP_Hp}
 If $X$ is a Banach space with the weak $p$-GPP, then $\Hp[X]$ is 1-monotonically bounded.
\end{theorem}

We shall make use of ultraproducts (see \cite{Hein} for an introduction into the topic).

\begin{lemma}\label{ultraproduct}
Suppose $\mathfrak U$ is an ultrafilter on a set $I$, $X_i, Y_i$ $(i \in I)$ are Banach spaces, and $T_i : X_i \to Y_i$ are linear operators, so that $\sup_{i \in I} \pi_p(T_i) < \infty$. Then $\pi_p(\prod_{\mathfrak U} T_i) \leq \lim_{\mathfrak U} \pi_p(T_i)$.
\end{lemma}

\begin{remark}
 The inequality in \Cref{ultraproduct} may be strict. For instance, take $p=2$, $I = \N$, $X_i = Y_i = \ell_2$ for any $i$, and suppose $\mathfrak U$ is a free ultrafilter on $\N$. Let $T_i = i^{-1/2} P_i$, where $P_i$ is an orthogonal projection of rank $i$. We know that $\pi_2(T_i) = 1$ for any $i$, yet $\prod_{\mathfrak U} T_i = 0$.
\end{remark}

\begin{proof}
By scaling, we can assume that $\lim_{\mathfrak U} \pi_p(T_i) = 1$.
We have to show that, for any $\vr > 0$ and $\overline{x}_1, ..., \overline{x}_N \in \prod_{\mathfrak U} X_i$ with $\|(\overline{x}_j)\|_{p,\textrm{weak}} < 1$, we have $\sum_j \| (\prod_{\mathfrak{U}} T_i) \overline{x}_j\|^p \leq (1+\vr)^{2p}$.
Note that, if $\overline{x}_j$ is represented by $(x_{ji})_{i \in I}$ ($x_{ji} \in X_i$), then $(\prod_{\mathfrak{U}} T_i) \overline{x}_j$ is the equivalence class of $(T_i x_{ji})_{i \in I}$, hence it suffices to show the existence of $U \in {\mathfrak{U}}$ so that $\sum_j \|T_i x_{ji}\|^p \leq (1+\vr)^{2p}$ for any $i \in U$.

To this end, for each $j$ find $(x_{ji})_{i \in I} \in (\sum_{i \in I} X_i)_\infty$ so that $\overline{x}_j$ is the equivalence class of $(x_{ji})_{i \in I}$ (modulo the class $N_{\mathfrak{U}}$ of all $(y_i)$ with $\lim_{\mathfrak{U}} \|y_i\| = 0$). Let $C = \max_j \sup_i \|x_{ji}\|$.
Denote by $S$ the unit ball of $\ell_{p'}^N$. Let $S_0$ be a finite subset of $S$ with the property that for any $\widetilde{\alpha} = (\alpha_j)_{j=1}^N \in S$ there exists $\widetilde{\beta} = (\beta_j)_{j=1}^N \in S_0$ so that $\max_j |\alpha_j - \beta_j| < \vr/(4N\max\{C,1\})$.

For $\widetilde{\beta} = (\beta_j) \in S_0$ find $y_{\widetilde{\beta}} = (y_{\widetilde{\beta}i}) \in N_{\mathfrak{U}}$ so that $\|\sum_j \beta_j x_{ji} - y_{\widetilde{\beta}i}\| < 1 + \vr/4$ for any $i$. The set $I_{\widetilde{\beta}} = \{i \in I : \|y_{\widetilde{\beta}i}\| < \vr/4\}$ belongs to ${\mathfrak{U}}$, hence so does $I_0 = \cap_{\widetilde{\beta} \in S_0} I_{\widetilde{\beta}}$. 
By the triangle inequality, for any $i \in I_0$ and $\widetilde{\beta} \in S_0$, we have $\big\| \sum_j \beta_j x_{ji} \big\| \leq 1 + \vr/2$. A further application of the triangle inequality gives us $\big\| \sum_j \alpha_j x_{ji} \big\| \leq 1 + \vr$ for any $\widetilde{\alpha} = (\alpha_j) \in S$ and $i \in I_0$.
Further, there exists $I_1 \in {\mathfrak{U}}$ so that $\pi_p(T_i) < 1+\vr$ for $i \in I_1$. From this we conclude that, for any $i \in U = I_0 \cap I_1$, 
$$
\big(\sum_j \|T_i x_{ji}\|^p \big)^{1/p} \leq \pi_p(T_i) \sup_{\widetilde{\alpha} \in S} \big\| \sum_j \alpha_j x_{ji} \big\| \leq (1+\vr)^2 ,
$$
and our task is complete.
\end{proof}

The following two lemmas are motivated by \cite[Section 4]{PMT}. It is instructive to compare the first of these with \cite[Lemma 4.8]{PMT}.

\begin{lemma}\label{weak* closure}
Suppose $X$ is a Banach space, $K$ is a weak$^*$-compact absolutely convex subset of $X^{**}$, and $L = K \cap X$.
Then $C = \{x^* \in X^* : h_L(x^*) \leq 1\}$ is the weak$^*$ closure of $A = \{x^* \in X^* : h_K(x^*) \leq 1\}$.
\end{lemma}

\begin{proof}
We use Bipolar Theorem to show that $C = \overline{A}^{w^*}$. In general, we shall denote the polar of $B \subset W$ ($W$ being a Banach space) by $B^\circ$:
$$
B^\circ = \big\{ w^* \in W^* : \sup_{w \in B}| \langle w^*, w \rangle | \leq 1 \big\} .
$$
If, moreover, $W$ is the dual of $W_*$, we denote by $^\circ B$ the pre-polar of $B$: $^\circ B = B^\circ \cap W_*$.
In this notation, $A = {}^\circ K$ and $C = L^\circ$. By Bipolar Theorem, 
$$
K = \overline{K}^{w^*} = \big( {}^\circ K\big)^\circ = A^\circ ,
$$
and therefore, $^\circ A = K \cap X = L$. Working with the pair $(X^*, X)$ gives us:
$$
\overline{A}^{w^*} = \big( {}^\circ A\big)^\circ = L^\circ . \qedhere
$$
\end{proof}

This permits us to manipulate upper estimates on functions from $H_{w^*}^\infty[X]$.
\begin{lemma}\label{compare}
Suppose $X$ is a Banach space, $K$ is a weak$^*$-compact absolutely convex subset of $X^{**}$, and $L = K \cap X$. If $f \in H_{w^*}^\infty[X]$, and $|f| \leq h_K$, then $|f| \leq h_L$.
\end{lemma}

\begin{proof}
Note that, if $\varphi, \psi$ are positively homogeneous functions $X^* \to [0,\infty)$, then $\varphi \leq \psi$ if and only if
$$
\big\{ x^* \in X^* : \psi(x^*) \leq 1 \big\} \subset \big\{ x^* \in X^* : \varphi(x^*) \leq 1 \big\} .
$$
Indeed, if $\varphi \leq \psi$, then the above inclusion must hold. Conversely, suppose the inclusion does hold. By positive homogeneity, $\psi(x^*) = t \psi(x^*/t)$ for $t>0$, hence
$$
\psi(x^*) = \inf \big\{ t>0 : \psi(x^*/t) \leq 1 \big\} ,
$$
and the same goes for $\varphi$. If $\psi(x^*/t) \leq 1$, then also $\varphi(x^*/t) \leq 1$, leading to $\psi(x^*) \geq \varphi(x^*)$.

So let $C = \{x^* \in X : h_L(x^*) \leq 1\}$, $A = \{x^* \in X : h_K(x^*) \leq 1\}$, and $B = \{x^* \in X : |f(x^*)| \leq 1\}$. We need to show that $A \subset B$ implies $C \subset B$.

By \Cref{weak* closure}, $C = \overline{A}^{w^*}$. As the sets in question are convex, Banach-Dieudonn\'e Theorem gives us $C = \overline{A}^{bw^*}$. Here $bw^*$ refers to the \emph{bounded weak$^*$} topology, see e.g.~\cite[Section 2.7]{Meg} for details.
By \cite{AD}, $B$ is $bw^*$-closed; as it contains $A$, it must also contain its $bw^*$ closure $C$.
\end{proof}

One also obtains similar results for operator-induced (as opposed to set-induced) functions.

\begin{corollary}\label{compare operators}
Suppose $X, Y$ are Banach spaces, and $S \in B(X^*,Y)$. Then there exists an absolutely convex $L \subset X$ so that the associated operator $R \in B(\ell_1(L),X)$ satisfies $\phi_{R^*} \leq \phi_S$, and $|f| \leq \phi_{R^*}$ for any $f \in H_{w^*}^\infty[X]$ satisfying $|f| \leq \phi_S$. If, in addition, $S \in \Pi_p(X^*,Y)$, then $L$ is $p$-limited.
\end{corollary}

\begin{proof}
Let $K = S^*(\ball(Y^*)) \subset X^{**}$, $L = K \cap X$, and $R:  \ell_1(L) \to X: \delta_x \mapsto x,$ $(x \in L)$. Then
\begin{align*}
    \phi_{R^*}(x^*) & = \|R^*x^*\| = \sup_{x \in L} |\langle x^*, x \rangle| \leq \sup_{x^{**} \in K} |\langle x^{**}, x^* \rangle| 
    \\  & = \sup_{y^* \in \ball(Y^*)} |\langle S^*y^*, x^* \rangle| = \|S x^*\| = \phi_S(x^*) . 
\end{align*}
By \Cref{phi versus h}, $\phi_S = h_K$. If $f \in H_{w^*}^\infty[X]$ and $|f| \leq \phi_S = h_K$, then $|f| \leq h_L = \phi_{R^*}$, due to \Cref{compare,phi versus h}.

For the final assertion, observe that $S^{**}$ is $p$-summing whenever $S$ is. Thus, by \cite[Theorem 3.1]{DP-p-lim}, $K$ is $p$-limited, and hence so is $L$.
\end{proof}

\begin{proof}[Proof of \Cref{wRNP_Hp}]
Given an increasing net $(f_i)_{i \in I} \subset \ball(\Hp[X])_+$, we shall find $f \in \ball(\Hp[X])_+$ so that $f \geq f_i$ for every $i$.
By \Cref{FBLp norm}, for each $i$, there exist a probability measure $\mu_i$ and an operator $T_i : X^* \to L_p(\mu_i)$ with $\iota_p(T_i) \leq 1$, and $\|T_i x^*\| \geq f_i(x^*)$ for any $x^* \in X^*$.

Consider the \emph{order filter} on $I$, which consists of sets $S \subset I$ for which there exists $i \in S$ so that $j \in S$ whenever $j \geq i$. An ultrafilter containing such a filter shall be called \emph{order consistent}.
Suppose ${\mathcal{U}}$ is an order consistent ultrafilter on $I$, then, by \Cref{ultraproduct}, $\prod_{\mathcal{U}} T_i : (X^*)^{\mathcal{U}} \to (L_p(\mu_i))^{\mathcal{U}}$ has $p$-summing norm not exceeding $\lim_{\mathcal U} \pi_p(T_i) \leq 1$. 
Let $S$ be the restriction of the aforementioned operator to $X^*$ (canonically embedded into its ultrapower), then $\pi_p(S) \leq 1$, and, for any $x^* \in X^*$, $\|S x^*\| = \lim_{{\mathcal{U}}} \|T_i x^*\| \geq \vee_i f_i(x^*)$. 
By \Cref{compare operators}, there exist $p$-limited set $L \subset X$, so that the associated $R \in B(\ell_1(L),X)$ has the following properties: $f_i \leq \phi_{R^*}$ for every $i$, and $\phi_{R^*} \leq \phi_S$. The latter inequality means that $\pi_p(R^*) \leq \pi_p(S)$. The weak $p$-GPP implies that $L$ is relatively compact, hence $R$ is compact and $\phi_{R^*} \in \ball(\Hp[X])$ by \Cref{compact_adjoint}.
\end{proof}

Our technique also allows us to investigate a question raised in \cite{LaTr}. Specifically, we prove the following.

\begin{theorem}\label{la-tr}
 If $X$ is a Banach space with the $p$-GPP, and $1 \leq p < \infty$, then for any $f \in \ball(\Hp[X])_+$ there exists $g \in \ball(\FBLp[X])_+$ such that $f \leq g$. As a consequence, $\Hp[X] = \nip[X]$.
\end{theorem}

In \Cref{Ip into Hp} we shall show that the $p$-GPP of $X$ is also necessary for $\nip[X]$ to be closed and to coincide with $\Hp[X]$.

\begin{proof}
 Consider $f \in \ball(\Hp[X])_+$. Find a probability measure $\mu$ and an operator $S : X^* \to L_p(\mu)$ with $\pi_p(S) \leq 1$, and $\|S x^*\| \geq f(x^*)$ for any $x^* \in X^*$.
 By \Cref{compare operators}, there exist $p$-limited set $L$ and $R \in B(\ell_1(L),X)$ so that $f \leq \phi_{R^*} \leq \phi_S$. From the last inequality, $\pi_p(R^*) \leq \pi_p(S)$. As $X$ has the $p$-GPP, $L$ is relatively $p$-compact, hence $R$ is $p$-compact. Thus, $\phi_{R^*} \in \ball(\FBLp[X])$ by \Cref{p-compactness}.
\end{proof}

As the $p$-GPP implies the weak $p$-GPP, the following is an immediate consequence of \Cref{wRNP_Hp,la-tr}.
\begin{corollary}\label{dual-with-RNP}
 If $X$ is a Banach space with the $p$-GPP, then both $\FBLp[X]$ and $\Ip[X]$ are 1-monotonically bounded.
\end{corollary}

\section{Necessary conditions for upper bound properties}\label{operators:general}

We say that $(Z_k)$ is a \emph{net finite dimensional resolution} (\emph{NFDR}) for $T$ if $(Z_k)_{k \in {\mathcal K}}$ is an increasing net of finite dimensional subspaces of $Z$, and $\cup_k Z_k$ is $1$-norming for $T^*(X^*)$ -- that is, the equality
$$
\|T^* x^*\| = \sup_k \sup \big\{ \langle T^* x^*, z \rangle : z \in \ball(Z_k) \big\} 
$$
holds for any $x^* \in X^*$. Further, let $T_k = T \big|_{Z_k}$. Note that an NDPR arises, for instance, when $(Z_k)_{k \in {\mathcal K}}$ is the net of all finite dimensional subspaces of $Z$, ordered by inclusion.
A simple observation is in order:

\begin{lemma}\label{l:simple properties}
In the above notation:
\begin{enumerate}
    \item For every $x^* \in X^*$, $\|T_k^* x^*\| \nearrow \|T^* x^*\|$.
    \item $\pi_p(T_k^*) \nearrow \pi_p(T^*)$.
\end{enumerate}
\end{lemma}

In item (2), $\pi_p(T^*)$ need not be finite: throughout this paper, we set $\pi_p(S) = \infty$ if $S$ is not $p$-summing.

For convenience, we extend the definition of the weak $p$-GPP to $p = \infty$ by saying that $X$ has the weak $\infty$-GPP if every operator into $X$ is compact, in other words, if $X$ is finite-dimensional. This convention agrees with the operator characterization of the weak $p$-GPP given in \Cref{weak pGP}.
\begin{theorem}\label{no compact no fatou}
If $1 \leq p \leq \infty$ and $X$ fails the weak $p$-GPP, then the lattices $\FBLp[X],$ $\Ip[X],$ and $\Hp[X]$ fail the $\lambda$-Fatou property, for any $\lambda > 0$.
\end{theorem}

\begin{proof}
Let $E$ be any of the three lattices in question. Fix a non-compact $T \in B(Z, X)$ such that $T^* \in \Pi_p(X^*,Z^*)$. By scaling, we can assume $\pi_p(T^*) = 1$, and hence $\|T\| \leq 1$. Let $(Z_k)$ be an NFDR for $T$.
As before, write $T_k = T \big|_{Z_k}$, $\phi(x^*) = \|T^* x^*\|$, $\phi_k(x^*) = \|T_k^* x^*\|$. Note that, by \Cref{fin rank ops}, $\phi_k \in \FBLp[X]$ for all $k$.
Pick $\mu > \lambda$, and let $f_k = \phi_k \wedge \mu \big|\delta_{x_0}\big|$, where $x_0 \in X$ has unit norm.
We claim that $\mu \big|\delta_{x_0}\big| = \vee_k f_k$ in $E$. As $\|\mu \big|\delta_{x_0}\big|\| = \mu > \lambda$, and $\|f_k\| \leq \|\phi_k\| \leq \|\phi\| = 1$ for every $k$, we shall be done once the claim is established. Thus, we shall show that if $g \in E$ is such that $g \geq f_k$ for every $k$, then $g \geq \mu \big|\delta_{x_0}\big|$.
In other words, $g(x^*) \geq \mu |\langle x^*, x_0 \rangle|$ for every norm one $x^* \in X^*$.

By \Cref{weak* continuity}, $\phi|_{\ball(X^*)}$ is discontinuous at $0$, so there exists $c > 0$ and a weak$^*$ null net $(x_\alpha^*) \subset \ball(X^*)$ such that $\phi(x_\alpha^*) > c$ for each $\alpha$. Thus, for each $\alpha$, there exists $K(\alpha)$ such that $\phi_k(x_\alpha^*) > c$ for $k \geq K(\alpha)$.

Fix a norm one $x^* \in X^*$ and pick $C > (1 + \mu)/c$. Then
$$
\phi_k(x^* + C x_\alpha^*) = \|T_k^*(x^* + C x_\alpha^*)\| \geq C \|T_k^* x_\alpha^*\| - \|T_k^* x^*\| ,
$$
and hence $\phi_k(x^* + C x_\alpha^*) \geq Cc-1$ for $k \geq K(\alpha)$. Now fix $\sigma > 0$, and find $\alpha_0$ such that $\big| \langle x_\alpha^* , x_0 \rangle \big| < \sigma/C$ for $\alpha \geq \alpha_0$. For such $\alpha$,
$$
\mu \big| \delta_{x_0} \big| (x^* + C x_\alpha^*) = \mu \big| \langle x^* + C x_\alpha^* , x_0 \rangle \big| > \mu \big| \langle x^* , x_0 \rangle \big| - \sigma .
$$
By our choice of $C$, for $\alpha \geq \alpha_0$ and $k \geq k(\alpha)$,
$$
f_k (x^* + C x_\alpha^*) \geq (Cc - 1) \wedge \big( \mu \big| \langle x^* , x_0 \rangle \big| - \sigma \big) = \mu \big| \langle x^* , x_0 \rangle \big| - \sigma .
$$
Consequently, for $\alpha \geq \alpha_0$, $g (x^* + C x_\alpha^*) \geq \mu \big| \langle x^* , x_0 \rangle \big| - \sigma$. Passing to the limit along the net, we conclude that $g (x^*) \geq \mu \big| \langle x^* , x_0 \rangle \big| - \sigma$. It remains to recall that $\sigma$ is arbitrarily small.
\end{proof}

\begin{corollary}\label{duals and Hp}
    If $X$ be a Banach space and $p \in [1,\infty)$, then: 
    \begin{enumerate}
        \item $\Hp[X]$ is monotonically bounded if and only if $\Hp[X]$ has the $\lambda$-Fatou property for some $($equivalently, any$)$ $\lambda$ if and only if $X$ has the weak $p$-GPP;
        \item $\Ip[X]$ has the Fatou property if and only if $X$ has the weak $p$-GPP.
    \end{enumerate}
\end{corollary}

\begin{proof}
    If $X$ has the weak $p$-GPP, then, by \Cref{wRNP_Hp}, $\Hp[X]$ is 1-monotonically bounded and therefore has the Fatou property. Since $\Ip[X]$ is an ideal in $\Hp[X]$, it follows that $\Ip[X]$ has the Fatou property as well. The converse follows from \Cref{no compact no fatou}.
\end{proof}

\begin{corollary}
For a Banach space $X$, the following are equivalent:
\begin{enumerate}
    \item $X$ is infinite dimensional.
    \item For any $\lambda > 0$, $\FBL^{(\infty)}[X]$ fails the $\lambda$-Fatou property.
    \item For some $\lambda > 0$, $\FBL^{(\infty)}[X]$ fails the $\lambda$-Fatou property.
\end{enumerate}
\end{corollary}

\begin{proof}
Clearly, $(2) \Rightarrow (3)$. The implication $(3) \Rightarrow (1)$ follows from the fact that, if $\dim X < \infty$, then $\FBL^{(\infty)}[X]$ can be identified with $C(\sphere(X^*))$, which has the Fatou property. For $(1) \Rightarrow (2)$, note that $X$ fails the weak $\infty$-GPP, and apply \Cref{no compact no fatou}.
\end{proof}

\begin{theorem}\label{failure of fatou}
If $1 \leq p < \infty$, and $X$ fails the $p$-GPP, then $\FBLp[X]$ fails the $\lambda$-Fatou property, for any $\lambda$, and $\Ip[X]$ is not monotonically bounded.
\end{theorem}

The proof relies on:

\begin{lemma}\label{lipschitz}
 If $P : \R^n \to \R$ is lattice polynomial, and $p \in [1,\infty)$, then there exists $K > 0$ so that, for any $t = (t_i)_{i=1}^n$ and $s = (s_i)_{i=1}^n$, we have $|P(t) - P(s)|^p \leq K \sum_i |t_i - s_i|^p$.
\end{lemma}

Here and below, a \emph{lattice polynomial} refers to a real-valued function on $\R^n$, obtained by repeated application of linear and lattice operations.

\begin{proof}
 By the standard H\"older Inequality, it suffices to show that for any $P$ there exists $K_0 = K_0(P)$ so that, for any $t,s$ as above, 
 \begin{equation}
  |P(t) - P(s)| \leq K_0 \sum_i |t_i - s_i| .
  \label{eq:lipschitz}
 \end{equation}
 This we establish by induction on the number of lattice operations involved in $P$.
 The base of induction reduces to considering $P(t) = t_i$, which clearly satisfies \eqref{eq:lipschitz}, with $K_0 = 1$.
 
 For the inductive step, it suffices to show that if $P_1, P_2$ satisfy \eqref{eq:lipschitz}, then so do $P_1 + P_2$, $P_1 \vee P_2$, and $a P_1$ for any $a \in \R$ (with a different constant $K_0$). This, however, is straightforward. For instance, for $P = P_1 \vee P_2$, note that $|P(t) - P(s)| \leq |P_1(t) - P_1(s)| + |P_2(t) - P_2(s)|$, hence $K_0(P) \leq K_0(P_1) + K_0(P_2)$.
\end{proof}

\Cref{lipschitz} directly implies:

\begin{corollary}\label{approximation}
Suppose $X_0$ is a dense subset of $X$. Denote by $\mathcal Y$ the family of all lattice expressions $F(\delta_{x_1}, \ldots, \delta_{x_n})$, with $x_1, \ldots, x_n \in X_0$; here $F : \R^n \to \R$ is a lattice polynomial. Then:
\begin{enumerate}
\item For any $\vr > 0$ and $f \in \FBLp[X]$ there exists $g \in \mathcal Y$ such that $\|f-g\| < \vr$. If $f$ is positive, then $g$ can be taken positive as well.
\item For any $\vr > 0$ and $f \in \Ip[X]_+$ there exist $g \in {\mathcal Y}_+$ and $h \in \Hp[X]_+$ such that $\|h\| < \vr$ and $f \leq g+h$.
\end{enumerate}
\end{corollary}

\begin{lemma}\label{go far to get nothing}
Suppose $X_0$ is dense in $X$. Then, for any $g \in \Ip[X]$ and $\vr > 0$, there exist $x_1, \ldots, x_m \in X_0$ so that, for every $x_1^*, \ldots, x_n^* \in \cap_{i=1}^m x_i^\perp$ with $\|(x_j^*)\|_{p,\textrm{weak}} \leq 1$, we have $\sum_j |g(x_j^*)|^p < \vr^p$.
\end{lemma}

\begin{proof}
We can and do assume $g \geq 0$. Fix $\vr > 0$, 
by \Cref{approximation}(2), there exist $g_1 = G\big(\delta_{x_1}, \ldots, \delta_{x_m}\big) \in \FVL[X]_+$ ($G$ is a lattice polynomial, $x_1, \ldots, x_m \in X_0$) and $g_2 \in \Hp[X]$ with $\|g_2\| < \vr$, such that $g \leq g_1 + g_2$. 
If $x_1^*, \ldots, x_n^* \in \cap_{i=1}^m x_i^\perp$, then, for every $j$, $g_1(x_j^*) = 0$, hence $g(x_j^*) \leq g_2(x_j^*)$. Therefore,
$$
\sum_j |g(x_j^*)|^p \leq \|g_2\|^p \|(x_j^*)\|_{p,\textrm{weak}}^p < \vr^p . \qedhere
$$
\end{proof}
This yields the following complement to \Cref{p-compactness}.
\begin{corollary}\label{Ip and operators}
    For an operator $T: Z \to X$, $T$ is $p$-compact if and only if $\phi_{T^*} \in \Ip[X]$.
\end{corollary}
\begin{proof}
    Combine the preceding lemma with \Cref{p-compactness,quasi p-nuc criteria}.
\end{proof}

\begin{proposition}\label{supremum in FBL}
Suppose $T \in B(Z,X)$ is not $p$-compact. Fix $x_0 \in X$, if $g \in \FBLp[X]$ satisfies $g \geq f = \big| \delta_{x_0} \big| \wedge \phi_{T^*}$, then $g \geq \big| \delta_{x_0} \big|$.
\end{proposition}

\begin{proof}
By scaling, we assume that $\pi_p(T^*|_E) > 1$ whenever $E \subset X^*$ is finite codimensional. Pick a norm one $x^* \in X^*$, and show that $g(x^*) \geq |\langle x^*, x_0\rangle|$. 

Pick $\delta > 0$, and write $g = g_1 + g_2$, where $g_1 = G(\delta_{x_1}, \ldots, \delta_{x_m})$, $G$ being a lattice polynomial, and $\|g_2\| < \delta$.
Find $x_1^*, \ldots, x_n^* \in \cap_{i=0}^m x_i^\perp$ so that $\|(x_j^*)\|_{p,{\textrm{weak}}} \leq 1$, and $\sum_j \alpha_j^p = 1$, where $\alpha_j = \|T^* x_j^*\|$.

Pick $C > \|T\| + \|x_0\| + 1$. For each $j$, $g(\alpha_j x^* + C x_j^*) = \alpha_j g_1(x^*) + g_2(\alpha_j x^* + C x_j^*)$, or, in other words,
$$
\big( g(\alpha_j x^* + C x_j^*) \big)_j - g_1(x^*) \big( \alpha_j \big)_j = \big( g_2(\alpha_j x^* + C x_j^*) \big)_j .
$$
The triangle inequality in $\ell_p^n$ gives us:
$$
\Big|\big\| \big( g(\alpha_j x^* + C x_j^*) \big)_j \big\|_p - g_1(x^*) \Big| \leq \big\| \big( g_2(\alpha_j x^* + C x_j^*) \big)_j \big\|_p 
$$
(here we use our scaling: $\big\| \big( \alpha_j \big)_j \big\|_p = 1$). Taking $\|x^*\| = 1$ into account,
$$
\big\| \big( \alpha_j x^* + C x_j^* \big)_j \big\|_{p, {\textrm{weak}}} \leq C \big\| \big( x_j^* \big)_j \big\|_{p, {\textrm{weak}}} + \big\| \big( \alpha_j \big)_j \big\|_p \leq C+1 ,
$$
hence $\big\| \big( g_2(\alpha_j x^* + C x_j^*) \big)_j \big\|_p < \delta (C+1)$, and therefore,
\begin{equation}
\Big|\big\| \big( g(\alpha_j x^* + C x_j^*) \big)_j \big\|_p - g_1(x^*) \Big| < (C+1)\delta .
\label{g g1 close}
\end{equation}

Now note that
$$
f(\alpha_j x^* + C x_j^*) = \alpha_j \big| \langle x^*, x_0 \rangle \big| \wedge \|T^* (\alpha_j x^* + C x_j^*)\| ,
$$
and
$$
\|T^* (\alpha_j x^* + C x_j^*)\| \geq C \|T^* x_j^*\| - \alpha_j \|T^* x^*\| = \alpha_j (C - \|T^*x^*\|) .
$$
By our choice of $C$, $C - \|T^*x^*\| > \|x_0\| + 1$, hence $f(\alpha_j x^* + C x_j^*) = \alpha_j \big| \langle x^*, x_0 \rangle \big|$. We clearly have
$$
\big\| \big( g(\alpha_j x^* + C x_j^*) \big)_j \big\|_p \geq \big\| \big( f(\alpha_j x^* + C x_j^*) \big)_j \big\|_p = \big\| \big( \alpha_j \big)_j \big\|_p \big| \langle x^*, x_0 \rangle \big| = \big| \langle x^*, x_0 \rangle \big| ,
$$
so from \eqref{g g1 close}, $g_1(x^*) \geq \big| \langle x^*, x_0 \rangle \big| - (C+1)\delta$. Hence,
$$
g(x^*) \geq g_1(x^*) - \|g_2\| \geq \big| \langle x^*, x_0 \rangle \big| - (C+2)\delta .
$$
To complete the proof, recall that $\delta$ can be arbitrarily small.
\end{proof}

\begin{proof}[Proof of \Cref{failure of fatou}]
Find $T \in B(Z,X)$ which is not $p$-compact, while $T^*$ is $p$-summing. By \Cref{quasi p-nuc criteria}, and scaling, we may assume that $\pi_p(T^*|_E) > 1$ for any $E \subset X^*$ of finite codimension.
Keep the earlier notation: $(Z_k)$ is an NFDR for $T$, $T_k = T|_{Z_k}$, $\phi_{T^*}(x^*) = \|T^* x^*\|$, $\phi_k(x^*) = \|T_k^* x^*\|$. 

(1) First we establish the failure of $\lambda$-Fatou for $\FBLp[X]$. Fix $\mu > \lambda \pi_p(T^*)$ and a norm one $x_0 \in X$. Let $f_k = \phi_k \wedge \mu \big| \delta_{x_0} \big|$; clearly $f_k \in \FBLp[X]$, and $\|f_k\| \leq \|\phi_k\| \leq \pi_p(T^*)$.
If $g \in \FBLp[X]$, and $g \geq f_k$ for any $k$, then $g \geq \mu \big| \delta_{x_0} \big| \wedge \phi_{T^*}$, hence, by \Cref{supremum in FBL}, $g \geq \mu \big| \delta_{x_0} \big|$, so $\|g\| \geq \mu > \lambda \pi_p(T^*)$.

(2) We show that $\Ip[X]$ is not monotonically bounded by proving that there is no $g \in \Ip[X]$ so that $g \geq \phi_k$ for every $k$.
Indeed, if such a $g$ exists, then $g \geq \sup_k \phi_k = \phi_{T^*}$. By \Cref{go far to get nothing}, there exist $x_1, \ldots, x_m \in X$ so that $\sum_j g(x_j^*)^p < 1/2$ whenever $x_1^*, \ldots, x_n^* \in E = \cap_i x_i^\perp$ satisfy $\|(x_j^*)\|_{p,{\textrm{weak}}} \leq 1$.
However, $\pi_p(T^*|_E) > 1$, hence we can find $x_1^*, \ldots, x_n^*$ as above so that $\sum_j \phi_{T^*}(x_j^*)^p > 1$, yielding the desired contradiction.
\end{proof}

The following is an immediate consequence of \Cref{dual-with-RNP} and \Cref{failure of fatou}.

\begin{corollary}\label{duals and FBLp}
If $X$ be a Banach space and $p \in [1,\infty)$, then: 
\begin{enumerate}
\item $\FBLp[X]$ is monotonically bounded if and only if $\FBLp[X]$ has the $\lambda$-Fatou property for some $($equivalently, any$)$ $\lambda$ if and only if $X$ has the $p$-GPP;
\item $\Ip[X]$ is monotonically bounded if and only if $X$ has the $p$-GPP.
\end{enumerate}
\end{corollary}

\begin{remark}\label{SFDR}
If the relevant operator $T$ in \Cref{no compact no fatou,failure of fatou} can be chosen to have NFDR, consisting of an increasing \emph{sequence} $(Z_k)$ of finite dimensional subspaces of $Z$, then the same proofs yield countable versions of the corresponding conclusions. More precisely, in the setting of \Cref{no compact no fatou}, the lattices $\FBLp[X],$ $\Ip[X],$ and $\Hp[X]$ fail the countable $\lambda$-Fatou property, for any $\lambda > 0$. In the setting of \Cref{failure of fatou}, $\FBLp[X]$ fails the countable $\lambda$-Fatou property, for any $\lambda$, and $\Ip[X]$ is not countably monotonically bounded. For specific examples, see \Cref{SFDR in action}.
\end{remark}

\section{Comparing free and related Banach lattices}\label{s:embeddings}
In this section, we use our operator approach to study embeddings between $\FBLp[X]$ and associated function lattices. We begin with the following lemma.
\begin{lemma}\label{copies l_2^n}
    Let $X$ be an infinite-dimensional Banach space, and $1 \leq p < \infty$. There exist $c > 0$, integers $N_k$ and contractions $u_k: \ell_1^{N_k} \to X$ such that for every finite-codimensional $F \subset X^*$ and every $k \geq 2 \, \codim F$, we have ${\boldsymbol{\kappa}}_k \pi_p(u_k^*|_F) \geq c$, where ${\boldsymbol{\kappa}}_k = (\pi_p(I_{\ell_2^k}))^{-1}$.
\end{lemma}
\begin{remark}\label{kappa}
    By \cite[Corollary 4.13]{DJT}, ${\boldsymbol{\kappa}}_k \sim_p k^{-1/2}$.
\end{remark}
\begin{proof}
    For $k \in \N$, use Dvoretzky's theorem to find $E_k \subset X^*$ which is $2$-isomorphic to $\ell_2^k$.
Further, find a contractive embedding $w_k : E_k \to \ell_\infty^{N_k}$ with $\|w_k^{-1}\| < 2$. Extend $w_k$ to a contraction $v_k : X^* \to \ell_\infty^{N_k}$.
By \cite{OjP}, there exists a contraction $u_k : \ell_1^{N_k} \to X$ such that, for every $x^* \in E_k$, $\|u_k^* x^*\| \geq \|x^*\|/2$.

If $F \subset X^*$ has finite codimension, and $k \geq 2 \, \codim F$, then $\dim (E_k \cap F) \geq k/2$, and 
\begin{align*}
{\boldsymbol{\kappa}}_k \pi_p(u_k^*|_{F}) \geq {\boldsymbol{\kappa}}_k \pi_p(u_k^*|_{E_k \cap F})
& \geq \frac{{\boldsymbol{\kappa}}_k}2 \pi_p(I_{E_k \cap F}) \geq \frac{{\boldsymbol{\kappa}}_k}4 \pi_p\big(I_{\ell_2^{\dim (E_k \cap F)}}\big) 
\\ &\geq \frac{{\boldsymbol{\kappa}}_k}{4 {\boldsymbol{\kappa}}_{\lceil k/2 \rceil}} \geq c , 
\end{align*}
with $c > 0$ being a constant depending only on $p$.
\end{proof}

\subsection{Embedding of $\FBLp[X]$ into $\nip[Y]$}\label{FBLp into Ip}
In \cite{LaTr}, it was shown that, if $X$ is a Banach space with an infinite-dimensional separable quotient, then, for $p \in [1, \infty)$, $\FBLp[X]$ is a proper subspace of $\Ip[X]$. 
Recently, OpenAI \cite{OpenAI} established that, under Continuum Hypothesis, there exist spaces with no such quotient (see also an exposition \cite{Argyros}). We strengthen the result of \cite{LaTr} by removing the separable quotient assumption.

For Banach spaces $X, Y$ and a bounded linear operator $i: X \to Y$, we consider the lattice homomorphism $\overline{i}: \Hp[X] \to \Hp[Y]$ defined by $[\overline{i} f](y^*) = f(i^* y^*)$. The argument of \cite{OiTTT-cor} shows that $i$ is injective if and only if $\overline{i}$ is. Note that, by \cite{OiTTT}, $\overline{i}(\FBLp[X]) \subset \FBLp[Y]$. 

\begin{theorem}\label{very irregular}
Let $X$ and $Y$ be infinite-dimensional Banach spaces, and let $i: X \to Y$ be an isometric embedding. Then, for $p \in [1,\infty)$, 
$\overline{i}(\FBLp[X])$ is not a regular sublattice of $\nip[Y]$.
\end{theorem}

In the particular case of $X=Y$, we have

\begin{corollary}
$\FBLp[X]$ coincides with $\Ip[X]$ if and only if either $X$ is finite dimensional, or $p=\infty$.
\end{corollary}

\begin{proof}[Proof of \Cref{very irregular}]
For $k \in \N$, fix $N_k$, ${\boldsymbol{\kappa}}_k$ and $u_k : \ell_1^{N_k} \to X$ given by \Cref{copies l_2^n}, and define $T : \ell_1 = (\oplus_k \ell_1^{N_k})_1 \to X : (\xi_k) \mapsto \sum_k {\boldsymbol{\kappa}}_k u_k \xi_k$. By \Cref{kappa}, $T$ is a compact contraction. However, it follows from \Cref{quasi p-nuc criteria} and \Cref{copies l_2^n} that $T$ is not $p$-compact. 

Let $T_n$ be the restriction of $T$ to $(\oplus_{k=1}^n \ell_1^{N_k})_1$.  Set $\phi_n(x^*) = \|T_n^* x^*\|$ and $\phi(x^*) = \|T^* x^*\|$. By \Cref{weak* continuity}, $\phi \in H_{w^*}[X]$, and it is easy to see that  $\phi_n \in \FVL[X]$. Pick $x \in X$ with $\|x\| > 1$; note that $\big |\delta_x \big|$ is not dominated by $\phi$ since $\phi(x^*) \leq 1$ for every norm-one $x^* \in X^*$. 
Let $f_n = \phi_n \wedge \big| \delta_x \big|$, in light of \Cref{supremum in FBL}, $\vee_n^{\FBLp[X]} f_n = \big| \delta_x \big|$. 

On the other hand, $\overline{i} \big| \delta_x \big| = \big| \delta_{ix} \big|$, and $\overline{i} \phi_n = \phi_n' \in \FBLp[Y]$, where $\phi_n'(y^*) = \|T_n^* i^* y^*\|$. Also, $\overline{i} \phi = \phi'$, where $\phi'(y^*) = \|T^* i^* y^*\|$. Thus $\phi' = \vee_n \phi_n'$ pointwise, and hence $\phi' \wedge \big|\delta_{ix}\big|$ is the pointwise supremum of $(\overline i f_n)_n$. Since $\phi' \wedge \big|\delta_{ix}\big| \in \Ip[Y]$, we have
$$
\vee_n^{\Ip[Y]} \overline i f_n = \phi' \wedge \big|\delta_{ix}\big| = \overline{i} \big( \phi \wedge \big|\delta_x\big| \big) \neq \overline{i} \big|\delta_x\big| , 
$$
by the injectivity of $\overline{i}$.
\end{proof}

\subsection{Embedding of $\Ip[X]$ into $\Hp[X]$}\label{FBLp vs Ip vs Hp}\label{Ip into Hp}
In \Cref{la-tr}, we showed that if $X$ has the $p$-GPP, then $\nip[X] = \Ip[X] = \Hp[X]$. We now prove that the $p$-GPP of $X$ is also necessary for both equalities to hold.

\begin{theorem}\label{Ip not Hp general}
For a Banach space $X$, the following statements are equivalent: 
\begin{enumerate}
\item $X$  has the $p$-GPP;
    \item $\nip[X]$ is closed in $\Hp[X]$;
    \item $\Ip[X] = \Hp[X]$.
\end{enumerate}
\end{theorem}

The proof requires:

\begin{proposition}\label{ideal is closed criterion}
The ideal $\nip[X]$ is closed in $\Hp[X]$ if and only if there exists a constant $C \geq 1$ such that for every $f \in \nip[X]$ there exists $g \in \FBLp[X]$ with $|f| \leq g$ and $\|g\| \leq C \|f\|$.
\end{proposition}

\begin{proof}
Suppose there exists a constant $C$ as above. We show that the sum of any absolutely convergent series in $\nip[X]$ belongs to $\nip[X]$. Indeed, suppose $f = \sum_n f_n$ and $\sum_n \|f_n\| < \infty$ (the sum exists in $\Hp[X]$).
For each $n$, find $g_n \in \FBLp[X]$ such that $|f_n| \leq g_n$ and $\|g_n\| \leq C \|f_n\|$, then $g = \sum_n g_n \in \FBLp[X]$, and $|f| \leq g$, so $f \in \nip[X]$.

Conversely, suppose $C$ as above does not exist. Then, for each $n$, we can find $f_n \in \nip[X]$ so that $\|f_n\| < 2^{-n}$, yet $\|g\| > 2^n$ whenever $g \in \FBLp[X]$ and $|f_n| \leq g$. Then $f = \sum_n f_n \in \Ip[X] \backslash \nip[X]$.
\end{proof}

\begin{proof}[Proof of \Cref{Ip not Hp general}]
\Cref{la-tr} establishes $(1) \Rightarrow (2) \vee (3)$. We focus on the converse.
Suppose $X$ fails the $p$-GPP (in particular, $X$ is infinite-dimensional).

$\lnot (1) \Rightarrow \lnot (2)$ Find $T \in B(Z,X)$ which is not $p$-compact and has $p$-summing adjoint.  For $x_0 \in X$, \Cref{supremum in FBL} implies that if $g \in \FBLp[X]$ satisfies $g \geq \big| \delta_{x_0} \big| \wedge \phi_{T^*} = f$, then $g \geq \big| \delta_{x_0} \big|$. Hence $\|g\| \geq \|x_0\|$, while $\|f\| \leq \|\phi_{T^*}\| = \pi_p(T^*)$. In light of \Cref{ideal is closed criterion}, $\nip[X]$ is not closed in $\Hp[X]$. 

$\lnot (1) \Rightarrow \lnot (3)$ If there exists a compact operator $T \in B(Z,X)$ which is not $p$-compact but has $p$-summing adjoint -- equivalently, if the failure of the $p$-GPP is witnessed by a relatively compact set $A \subset X$ -- then $\phi_{T^*} \in \Hp[X] \setminus \Ip[X]$ by \Cref{compact_adjoint} and \Cref{Ip and operators}. We may therefore assume that there exists a $p$-limited, not relatively compact set $A \subset X$. 

Find $\varepsilon > 0$ and $(x_k)_{k \in \N} \subset A$ so that $\|x_i - x_j\| \geq \varepsilon$, $(i \neq j)$, and set $d = \min\{\varepsilon/3, 1/2\}$. Then, for every finite-dimensional $E \subset X$, there exists $N_E \in \N$ such that $\dist(x_k, E) \geq d$ for all $k \geq N_E$. Let ${\boldsymbol{\kappa}}_k$ and $u_k$ be as in \Cref{copies l_2^n} and define $f(x^*) = \sup_k \big(|x^*(x_k)| \wedge {\boldsymbol{\kappa}}_k \|u_k^* x^*\|\big)$. We shall show that $f \in \Hp[X] \setminus \Ip[X]$.

Let $f_n(x^*) = \sup_{k \leq n} \big(|x^*(x_k)| \wedge {\boldsymbol{\kappa}}_k \|u_k^* x^*\|\big)$. It is easy to see that $f_n \in \FVL[X]$ for all $n$. Since $\|f_n - f_{n+1}\|_{\infty} \leq {\boldsymbol{\kappa}}_{n + 1}$, \Cref{kappa} yields $f \in H_{w^*}[X]$. Note also that $p$-limitedness of $A$ implies that the associated operator $R: \ell_1(A) \to X$ has $p$-summing adjoint, hence $\|h_A\|_p = \|\phi_{R^*}\|_p < \infty$. Since $f \leq h_A$, we have $\|f\|_p < \infty$, and hence $f \in \Hp[X]$. It remains to show that $f \notin \Ip[X]$.

By \Cref{go far to get nothing}, it suffices to verify that, for some $C > 0$, if $E \subset X$ is finite-dimensional, then there exists a finite sequence $(y_j^*) \subset E^\perp$ so that $\|(y_j^*)\|_{p, \mathrm{weak}} \leq C+1$ and $\sum_{j} f(y_j^*)^p \geq d^p$. 

Fix a finite-dimensional $E \subset X$ and $n > \max\{2(\dim E + 1), N_E\}$ so that ${\boldsymbol{\kappa}}_n \leq 1/2$, and find $x^* \in E^\perp$ with $\|x^*\| = 1$ and $|x^*(x_n)| \geq d$. \Cref{copies l_2^n} yields a constant $C > 0$, depending only on $p$, such that there exist $x_1^*, ..., x_l^* \in \spn\{E, x_n\}^\perp$ with $\|(x_j^*)\|_{p, \mathrm{weak}} \leq C$ and $\sum_j\alpha_j^p = 1$, here $\alpha_j = {\boldsymbol{\kappa}}_n \|u_n^* x_j^*\|$. Denote $y_j^* = \alpha_j x^* + x_j^*$; clearly, $y_j^* \in E^\perp$ and $\|(y_j^*)\|_{p, \mathrm{weak}} \leq C + 1.$ Note also that 
$$
{\boldsymbol{\kappa}}_n \|u_n^* y_j^*\| \geq {\boldsymbol{\kappa}}_n (\big| \|u_n^* x_j^*\| - \alpha_j \|u_n^* x^*\|\big|) \geq {\boldsymbol{\kappa}}_n |{\boldsymbol{\kappa}}_n^{-1} \alpha_j - \alpha_j| \geq \alpha_j d,
$$ by our choice of $n$ and $d$. Therefore,
$$
    f(y_j^*) \geq |y_j^*(x_n)| \wedge {\boldsymbol{\kappa}}_n \|u_n^* y_j^*\| = \alpha_j |x^*(x_n)| \wedge {\boldsymbol{\kappa}}_n \|u_n^* y_j^*\| \geq \alpha_j d,
$$ and hence $\sum_j f(y_j^*)^p \geq d^p$, as desired. 
\end{proof}

\subsection{The closure of $I_{w^*}$ in $\| \cdot \|_\infty$} \label{sup-norm}
The situation changes drastically if we consider the $\sup$ norm. For the sake of brevity, we denote $\ball(X^*)$, equipped with the weak$^*$ topology, by $K$. 
All the function spaces we are dealing with -- in particular, $I_{w^*}[X] \subset \nip[X] \subset \Ip[X] \subset \Hp[X]$ are linear subspaces of $C(K)$; we equip the latter with its usual $\sup$ norm $\| \cdot \|_\infty$. 

\begin{proposition}\label{denseness}
For any Banach space $X$, and $1 \leq p \leq \infty$, 
$$ \overline{\ball(I_{w^*}[X])}^{\| \cdot \|_\infty} = \overline{\ball(\Ip[X])}^{\| \cdot \|_\infty} = \ball(\Hp[X]) . $$
Here, $\ball(I_{w^*}[X])$ consists of those $f \in I_{w^*}[X]$ for which $\|f\|_p \leq 1$, where $\| \cdot \|_p$ is the norm inherited from $\Hp[X]$.
\end{proposition}

\begin{proof}
For $p = \infty$, this result is contained in \cite{OiTTT}. Henceforth, we assume $p \in [1,\infty)$. 

Clearly $\ball(\Ip[X]) \subset \ball(\Hp[X])$, and, since the $\| \cdot \|_\infty$ limits do not increase the $\| \cdot \|_p$ norm, $\overline{\ball(\Ip[X])}^{\| \cdot \|_\infty} \subset \ball(\Hp[X])$. Our goal is to establish $ \overline{\ball(I_{w^*}[X])}^{\| \cdot \|_\infty} \supset \ball(\Hp[X])$.

Denote by $A$ the set of all positively homogeneous functions $g \in C(K)$, for which there exist $x_1, \ldots, x_N \in X$ so that $\sum_{i=1}^N \|x_i\|^p \leq 1$, and $|g| \leq \big( \sum_i |\delta_{x_i}|^p \big)^{1/p}$. 
Further, let $A^p = \{h \in C(K) : |h|^{1/p} \in A\}$. We view $A^p$ as a subset of $W$ -- the sublattice of $C(K)$ consisting of functions $g$ which satisfy $g(sx^*) = s^p g(x^*)$, for any $s \geq 0$ and $x^* \in X^*$. 
In other words, $h \in W$ lies in $A^p$ if and only if there exist $x_1, \ldots, x_N \in X$ so that $\sum_{i=1}^N \|x_i\|^p < 1$, and $|h| \leq \sum_i  |\delta_{x_i}|^p$. 

First deal with positive functions: show that, if $f \in \Hp[X]_+$, and $f^p$ doesn't belong to the $\| \cdot \|_\infty$-closure $A^p$, then $\|f\|_p \geq 1$. 

It is easy to see that $A^p$ is solid and convex. Therefore, by \cite{OiTu}, there exists $\phi \in W^*_+$ such that $\langle \phi , f^p \rangle \geq 1$, and 
\begin{equation}
\langle \phi , g \rangle \leq 1 \, \, {\textrm{for  any  }} \, \, g \in A^p .
\label{eq:action on Ap}
\end{equation}
Denote by $j$ the embedding $W \to C(K)$. Then $j^* : M(K) \to W^*$ is a weak$^*$ continuous quotient map ($M(K) = C(K)^*$ is the space of finite Radon signed measures on $K$), hence we can find $\mu \in M(K)$ so that $j^* \mu = \phi$, and $\|\mu\| = \|\phi\|$.
Write $\mu = \mu_+ - \mu_-$, then $j^* \mu_+ \geq \phi$, hence $\|\mu_+\| \geq \|\phi\| = \|\mu\|$. On the other hand, $M(K)$ is an AL-space, hence $\|\mu\| = \|\mu_+\| + \|\mu_-\|$. This is only possible if $\mu_- = 0$, so $\phi$ is implemented by a positive Radon measure $\mu$.

Fix $\vr > 0$. For each $x^* \in K$, there exists an open set $U(x^*) \subset K$ so that $|f(x^*)^p - f(y^*)^p| \leq \vr/(2\|\mu\|)$ whenever $y^* \in U(x^*)$. Moreover, for this $U(x^*)$ there exist $x_1, \ldots, x_M \in X$ (depending on $x^*$) so that
$$
\big\{ y^* \in K : \max_i | \langle y^*, x_i \rangle | < 1 \big\} \subset U(x^*) .
$$
The set on the left is convex, and therefore, we can assume that $U(x^*)$ is convex for any $x^*$. Such sets form an open cover of $K$, from which we can extract a finite subcover, say $U(x_1^*), \ldots, U(x_L^*)$.
Then $U(x_1^*)$, $U(x_2^*) \backslash U(x_1^*)$, $\cdots$, $U(x_L^*) \backslash \cup_{j=1}^{L-1} U(x_j^*)$ are disjoint Borel sets covering $K$. By tossing out sets of measure $0$, we obtain disjoint Borel sets $V_1, \ldots, V_N \subset K$ s.t. (i) $\mu(V_i) > 0$ for any $i$, and $\mu(K \backslash (\cup_i V_i)) = 0$, and (ii) for any $i$, $|f(y^*)^p - f(z^*)^p| < \vr/\|\mu\|$ whenever $y^*, z^* \in \overline{{\textrm{conv}}(V_i)}$. 

Let $t_i = \mu(V_i)$ and $\nu_i = t_i^{-1} \chi_{V_i} \mu$ (so $\nu_i$ is a probability measure supported on $V_i$). By \cite[Theorem 1.1]{Phelps}, any $V_i$ possesses a barycenter -- that is, a (unique) point $x_i^* \in \overline{{\textrm{conv}}(V_i)}$ so that $\langle x_i^*, x \rangle = \int \langle x^*, x \rangle d\nu_i(x^*)$ for any $x \in X$. 

Plugging $g = |\delta_x|^p$ ($x \in \sphere(X)$) into \eqref{eq:action on Ap}, we obtain:
\begin{align*}
1 
& 
\geq \big\langle \mu, |\delta_x|^p \big\rangle = \int_K | \langle x^*, x \rangle|^p \, d\mu(x^*) = \sum_i t_i \int | \langle x^*, x \rangle|^p \, d\nu_i(x^*)
\\ &
\geq \sum_i t_i \Big| \int \langle x^*, x \rangle  \, d\nu_i(x^*) \Big|^p = \sum_i t_i \big| \langle x_i^*, x \rangle \big|^p  = \sum_i \big| \langle t_i^{1/p} x_i^*, x \rangle \big|^p 
\end{align*}
or equivalently, $\|(t_i^{1/p} x_i^*)\|_{p,{\textrm{weak}}} \leq 1$.

On the other hand, for each $i$ and $x^* \in V_i$, we have $f(x^*)^p \leq f(x_i^*)^p + \vr/\|\mu\|$, hence
\begin{align*}
1 
&
\leq \int f(x^*)^p \, d\mu(x^*) = \sum_i t_i \int_{V_i} f(x^*)^p \, d\nu_i(x^*) 
\\ &
\leq \sum_i t_i \Big( f(x_i^*)^p + \frac{\vr}{\|\mu\|} \Big) = \sum_i f\big( t_i^{1/p} x_i^* \big)^p + \vr ,
\end{align*}
and so, $\|f\|_p^p \geq \sum_i f\big( t_i^{1/p} x_i^* \big)^p \geq 1-\vr$, and we are done, since $\vr>0$ is arbitrary.

For a general $f \in \ball(\Hp[X])$, write $f = f_+ - f_-$. We have shown that, for any $\vr > 0$, there exists a positive $g \in A \subset \ball(I_{w^*}[X])$ so that $\| |f| - g \|_\infty < \vr$.
Now let $h = f_+ \wedge g - f_- \wedge g$. Clearly $|h| = f_+ \wedge g + f_- \wedge g \leq g$, hence $h \in \ball(I_{w^*}[X])$. Also, $$|f - h| = |f_+ - f_+ \wedge g| + |f_- - f_- \wedge g| \leq ||f|-g|,$$ so $\|f - h\|_\infty < \vr$.
\end{proof}

\section{Examples: spaces possessing, or failing, the (weak) $p$-GPP}\label{examples}

\subsection{Classes of spaces possessing the (weak) $p$-GPP}\label{has GPOP}

We begin by formulating a sufficient condition for the $p$-GPP in terms of the Radon-Nikod\'ym property (RNP). For information about the RNP, see e.g.~\cite{DU} or \cite[Chapter 5]{BL}.

The case $p = 1$ of the following proposition was established in \cite{DP-p-lim}. 
\begin{proposition}\label{RNP vs GP}
    If $X$ is a dual Banach space with the RNP, then it has the $p$-GPP for all $1 \leq p < \infty$.
\end{proposition}

\begin{proof}
    Fix $T: Z \to X$ with $p$-summing adjoint and let $S = T^*|_{X_*}$ (here $X_*$ denotes a predual of $X$). For an isometric embedding $i: Z^* \to \ell_{\infty}(\Gamma)$, $iS$ is strictly $p$-integral, hence $p$-nuclear by \cite[Theorem II.4]{Car}. Thus $S$ is quasi-$p$-nuclear, so $T = S^*|_Z$ is $p$-compact.
\end{proof}

\begin{corollary}\label{good cases}
If $X$ is a separable dual space or $X$ is reflexive, then $\FBLp[X]$, $\Ip[X]$ and $\Hp[X]$ are 1-monotonically bounded.
\end{corollary}

\begin{proof}
In both cases, $X$ has the RNP (see, for instance, \cite{DU} or \cite[Chapter 5]{BL}).
\end{proof}

\begin{remark}\label{lattice case}
By \cite[Corollary 5.4.21]{M-N}, a separable Banach lattice $X$ has the RNP if and only if it is a dual of a Banach lattice. Thus, if $X$ is a separable Banach lattice with the RNP, then it possesses the $p$-GPP for any $p \in [1,\infty)$, and the lattices $\FBLp[X], \Ip[X], \Hp[X]$ are 1-monotonically bounded.
\end{remark}

We can also give a criterion for the (weak) $p$-GPP which works for a larger class of spaces.
\begin{corollary}
Suppose $Y$ is a Banach space, and $Y^*$ is isomorphic to a complemented subspace of a Banach lattice. Then the following are equivalent:
\begin{enumerate}
    \item $Y$ does not contain a copy of $\ell_1$.
    \item $Y^*$ does not contain a copy of $L_1(0,1)$.
    \item $Y^*$ has the weak $p$-GPP for some $($equivalently, all$)$ $p \in [1,\infty)$.
    \item $Y^*$ has the $p$-GPP for some $($equivalently, all$)$ $p \in [1,\infty)$.
    \item $Y^*$ has the RNP.
\end{enumerate}
\end{corollary}

\begin{proof}
(1) $\Leftrightarrow$ (2) comes from \cite{Hag73}. 
The implications (1) $\Leftrightarrow$ (3)  and (4) $\Rightarrow$ (3) are discussed in \Cref{weak PG-p and l1} and \Cref{properties of GPOP}, respectively, (5) $\Rightarrow$ (4) is given by \Cref{RNP vs GP}.
Finally, \cite{GS} establishes (1) $\Rightarrow$ (5).
\end{proof}

\subsection{Complete Continuity Property versus weak $p$-GPP}\label{no compactness}

In this and the following subsection, we provide concrete examples of Banach spaces failing the (weak) $p$-GPP, which induce function spaces failing the Fatou property. To describe operators witnessing failure of the (weak) $p$-GPP, we need a simple lemma.

\begin{lemma}\label{id*}
Suppose $\mu$ is a probability measure, and $1 < q \leq \infty$. Let $id : L_q(\mu) \to L_1(\mu)$ be the formal identity, then $id^*$ is the formal identity $L_\infty(\mu) \to L_{q'}(\mu)$; consequently, $\pi_{q'}(id^*) = 1$.
\end{lemma}

Recall that a Banach space $X$ is said to have the \emph{Complete Continuity Property} (\emph{CCP}), also referred to as the Compact Range Property, if every operator $T: L_1 \to X$ is Dunford--Pettis (in other words, completely continuous).

\begin{proposition}\label{no CCP}
If a Banach space $X$ fails the CCP, and $1 \leq p < \infty$, then $X$ fails the weak $p$-GPP.
\end{proposition}

\begin{proof}
Find $\widetilde{T} \in B(L_1(0,1), X)$ which is not Dunford-Pettis, and let $T = \widetilde{T} \circ id$, where $id : L_{p'} (0,1) \to L_1(0,1)$ is the canonical embedding. By \Cref{id*}, $\pi_p(T^*) \leq \| \widetilde{T}\|$. By \cite[Proposition 1]{bourgain-DP}, $T$ fails to be compact.
\end{proof}

Combine the preceding result with \Cref{no compact no fatou} to obtain:

\begin{corollary}
If $X$ fails the CCP, and $1 \leq p < \infty$, then $\FBLp[X]$, $\Ip[X]$, and $\Hp[X]$ fail the $\lambda$-Fatou property, for any $\lambda > 0$.
\end{corollary}

\begin{corollary}\label{c0 or L1}
 If $X$ contains either $c_0$ or $L_1(0,1)$, and $1 \leq p < \infty$, then $\FBLp[X]$, $\Ip[X]$, and $\Hp[X]$ fail the $\lambda$-Fatou property, for any $\lambda > 0$.
\end{corollary}
\begin{proof}
    In both cases, $X$ fails the CCP (see, for example, \cite[p. 76]{Talagrand}).
\end{proof}

For $\FBL[X]$, \Cref{c0 or L1} was claimed in \cite{ABT}. However, for $X$ non-separable, that proof contains a gap: it relies on $\overline{i}(\FBL[Y])$ being a regular sublattice of $\FBL[X]$ whenever $i : Y \to X$ is an isometric embedding. This is false in general, see \cite{PMT}; however, this preprint also fixes the abovementioned gap.

\begin{remark}
    It is well known that $X^*$ has the CCP if and only if $X$ contains no copy of $\ell_1$. Hence, the CCP and the weak $p$-GPP are equivalent for dual Banach spaces. We refer to \cite[Chapter 7]{Talagrand} for a detailed discussion of the CCP.
\end{remark}

\subsection{Examples of spaces failing the $p$-GPP}\label{no m-b}\label{dual persistence}\label{james tree}
Finally we provide new examples of Banach spaces failing the $p$-GPP. Some of these spaces nevertheless possess the weak $p$-GPP, thus answering a question from \cite{AM}.

Recall that the \emph{James Tree space} $JT$ is defined as the completion of all finitely supported double indexed sequences $(a_{ni})_{n \geq 0, 1 \leq i \leq 2^n}$, arranged as a tree in the canonical way -- namely, $(n+1,2i-1)$ and $(n+1,2i)$ are direct descendants of $(n,i)$. The norm of $a = (a_{ni})$ is
\begin{equation}
	\|a\| = \big( \sup \sum_k \big| \sum_{(n,i) \in S_k} a_{ni} \big|^2 \big)^{1/2} ,
	\label{define:JT}
\end{equation}
where the supremum runs over all finite collections of disjoint segments $(S_k)$. 
The canonical normalized basis of $JT$, which we denote by $(e_{ni}^*)$, is boundedly complete (see e.g.~\cite[Proposition 15.4.4, Remark 15.4.5]{AK}). Denoting the biorthogonal functionals by $(e_{ni})$, we see that $JT$ is the dual of their span, which we denote by $JT_*$. We refer the reader to \cite{LiSt} for more information on $JT_*$.

\begin{proposition}\label{JT*-Fatou}
The spaces $JT_*$ and $JT^{(n)}$ $(n$ odd$)$ fail the $p$-GPP for $p \geq 2$.
\end{proposition}

\begin{remark}\label{JT_* GPOP}
If $m \in \{2,4,\ldots\}$, then, by \cite{LiSt}, $JT^{(m)}$ is isomorphic to $JT \oplus \ell_2(\Gamma)$. By \cite{LiSt} again, any bounded sequence in $JT$ has a weakly Cauchy subsequence, hence the same is true for $JT^{(m)}$. Thus, for $n$ odd, $JT^{(n-1)}$ does not contain a copy of $\ell_1$, hence $JT^{(n)}$ has the weak $p$-GPP for any $p$. 
Moreover, since the weak $p$-GPP is inherited by closed subspaces (\Cref{inherited}), it follows that $JT_*$ also has the weak $p$-GPP. However, as noted above, these spaces fail the $p$-GPP for $p \geq 2$. We thus answer a question of \cite{AM} by showing that the $p$-GPP is strictly stronger than the weak $p$-GPP.
\end{remark}

Combining \Cref{JT*-Fatou} with \Cref{failure of fatou}, we immediately obtain:

\begin{corollary}\label{JT*-Fatou-lattices}
If $X$ stands for either $JT_*$ or $JT^{(n)}$ $(n$ odd$)$, then, for $p \geq 2$, $\FBLp[X]$ fails the $\lambda$-Fatou property, for any $\lambda > 0$, and $\Ip[X]$ is not monotonically bounded.
\end{corollary}

\begin{remark}\label{bourgain}
In the notation of the preceding corollary, $\Ip[X]$ has the Fatou property, by \Cref{duals and Hp}, however, it is not monotonically bounded, by \Cref{JT*-Fatou-lattices}; thus, the latter property is strictly stronger on the class of function spaces associated with free lattices. 
%
\end{remark}

\begin{proof}[Proof of \Cref{JT*-Fatou}]
By \cite[Proposition 3.6]{DP-p-lim}, if a Banach space $X$ fails the $p$-GPP, then so does $X^{**}$. It therefore suffices to establish the failure of the $p$-GPP in $JT_*$.

Let $\mu$ be the canonical probability measure on the Cantor set $\Delta = \{0,1\}^{\N}$ -- that is, $\mu = \nu^{\otimes \N}$, where $\nu$ is the measure on $\{0,1\}$, assigning $1/2$ to each of the two points.
We shall find $\widetilde{T} \in B(L_1(\mu), JT_*)$ and define $T = \widetilde{T} \circ id$, where $id : L_{p'}(\mu) \to L_1(\mu)$. Then $T^*$ is $p$-summing; by \Cref{quasi p-nuc criteria}, it suffices to show that $T$ fails to be dual $p$-persistent.

Consider $R : JT \to C(\Delta) : e_{ni}^* \mapsto h_{ni}$, where $h_{ni}$ is a ``Haar function''.
More specifically, consider the branch leading from the root $(0,1)$ to $(n,i)$, consisting of pairs $(k-1,j_k)$, with $1 \leq k \leq n$ and $j_k = 2j_{k-1} - 1 + a_k$, with $a_k \in \{0,1\}$. For $\sigma = (\sigma_k)_{k \in \N} \in \Delta = \{0,1\}^\N$, let $h_{ni}(\sigma) = 1$ if $\sigma_k = a_k$ for $1 \leq k \leq n$; otherwise, set $h_{ni}(\sigma) = 0$. This can also be re-stated in terms of the binary expansion of $i-1$.
	
By \cite[p.~92]{LiSt}, $R$ is contractive. Let $\widetilde{T} = R^*\big|_{L_1(\mu)}$, and $x_{ni} = R^* h_{ni}^*$, where $h_{ni}^* = 2^n h_{ni} \in L_1(\mu) \subset C(\Delta)^*$.
Then $\langle x_{ni} , e^*_{st} \rangle = 2^n \langle h_{ni} , h_{st} \rangle$ equals to:
	\begin{itemize}
		\item $1$ if $(s,t)$ is a predecessor of $(n,i)$, or $(n,i)$ itself.
		\item $2^{n-s}$ if $(s,t)$ is a successor of $(n,i)$.
		\item $0$ if $(n,i)$ and $(s,t)$ are not on the same branch.
	\end{itemize}
By \cite[Theorem 1]{LiSt} (invoked on p.~92 of that paper, where the notation $y^*_{ni}$ is used), $x_{ni} \in JT_*$; so $\widetilde{T}$ takes values in $JT_*$. Let $T = \widetilde{T} \circ id$, then
$$
\langle h_{ni}^* , \widetilde{T}^* e^*_{st} \rangle = \langle h_{ni}^* , h_{st} \rangle \, \, {\textrm{ for \,  any }}  \, \, (n,i) , (s,t) ,
$$
hence $\widetilde{T}^* e^*_{st} = h_{st} \in C(\Delta) \subset L_\infty(\mu)$, and therefore, $T^* e^*_{st} = h_{st} \in L_p(\mu)$. For future use, note that $\|T^* e^*_{st}\| = 2^{-s/p}$.

Let $X_0$ the (dense) subset of $JT_*$, consisting of finitely supported vectors. We shall show that $T$ is dual $p$-persistent relative to $X_0$, with constant $1$; \Cref{quasi p-nuc criteria} will then show that $T$ is not $p$-compact.

For $y_1, ..., y_M \in X_0$ we can find $s$ so large that $\langle e_{st}^*, y_i \rangle = 0$ for every $i \in \{1, \ldots, M\}$ and $t \in \{1, \ldots, 2^s\}$. As no two nodes $(s,t)$ (for this $s$) can belong to the same segment, we have
$$\big\| \sum_{t=1}^{2^s} \alpha_t e^*_{st} \big\| = \big( \sum_{t=1}^{2^s} |\alpha_t|^2 \big)^{1/2}.$$
Consequently, if $\sum_j |\alpha_t|^{p'} \leq 1$, then
$$
\big\| \sum_{t=1}^{2^s} \alpha_t e^*_{st} \big\| = \big( \sum_{t=1}^{2^s} |\alpha_t|^2 \big)^{1/2} \leq 1 
$$
(here we rely on $p' \leq 2$), hence $\| (e^*_{st})_{t=1}^{2^s}\|_{p, {\textrm{weak}}} = 1$. 
On the other hand, $\|T^* e^*_{st}\| = 2^{-s/p}$, hence $\sum_t \|T^* e^*_{st}\|^p = 1$.
\end{proof}

We next present a class of spaces failing the $1$-GPP. These examples highlight the connection between the $1$-GPP and the RNP (see \Cref{has GPOP}).
\begin{proposition}\label{L_1 no RNP}
    If $X$ is a subspace of $L_1(\mu)$ failing the RNP, then $X$ fails the $1$-GPP.
\end{proposition}

\begin{proof}
Let $j: X \to L_1(\mu)$ be the canonical embedding. Find $\widetilde{T} : L_1(0,1) \to X$ which is not representable. Let $id : L_\infty(0,1) \to L_1(0,1)$ be the formal identity, and $T = \widetilde{T} \circ id$. By \Cref{id*}, $T$ has 1-summing adjoint.
By \cite[C3]{DeFl}, $T$ is not nuclear, hence neither is $jT$, due to \cite[Theorem III.3]{SteRe}. As $L_\infty^*$ has the Approximation Property, \cite[5.9]{DeFl} implies that $(jT)^* = T^*j^*$ is not nuclear. Applying \cite[Theorem III.3]{SteRe} again, we see that $T^* j^*$ is not quasi-nuclear, hence neither is $T^*$. By \Cref{p-compactness}, $T$ is not $1$-compact, as desired.
\end{proof}

\begin{remark}\label{Bourgain-Rosenthal}
Recall that a Banach space has the Schur property if every relatively weakly compact set is relatively compact. By \cite[Proposition 2.1]{DP-p-lim}, every $p$-limited set is relatively weakly compact. Hence every Banach space with the Schur property has the weak $p$-GPP for all $p \in [1,\infty)$. In \cite{BouRo}, a subspace $X$ of $L_1(0,1)$ with the Schur property but without the RNP is constructed. It follows from the above observation and \Cref{L_1 no RNP} that $X$ has the weak $1$-GPP but fails the $1$-GPP.
\end{remark}
\begin{remark}\label{SFDR in action}
The proofs in the two preceding subsections relied on an operator $T = \widetilde{T} \circ id$ (where $id : L_{p'}(0,1) \to L_1(0,1)$ is the formal identity) failing to be ($p$)-compact, while having a $p$-summing adjoint. One can observe that such a $T$ has an NFDR consisting of an increasing sequence $(Z_k)$. 
Indeed, $id^*$, and consequently $T^*$, takes values in $L_p(0,1)$, which is separable. We can select $(Z_k) \subset L_{p'}(0,1)$ so that $\cup_k Z_k$ is norming for the range of $T$.

In light of \Cref{SFDR}, the results of \Cref{no compactness} can be sharpened: if $X$ fails the CCP, and $1 \leq p < \infty$, then $\FBLp[X]$, $\Ip[X]$, $\Hp[X]$ fail the countable $\lambda$-Fatou property, for any $\lambda > 0$.
In particular this happens when $X$ contains $c_0$ or $L_1(0,1)$.
Similarly, \Cref{james tree} shows that, for $p \in [2,\infty)$ the lattices $\FBLp[JT_*]$ and $\FBLp[JT^{(n)}]$ ($n$ odd) fail the countable $\lambda$-Fatou property, for any $\lambda > 0$; $\Ip[JT_*]$ and $\Ip[JT^{(n)}]$ fail to be countably monotonically bounded. Similar results hold for subspaces of $L_1(\mu)$ failing the RNP and $p = 1$.
\end{remark}

\end{document}